\documentclass[12pt]{amsart}
\usepackage[T2A]{fontenc}
\usepackage[english]{babel}
\usepackage{amsaddr}
\usepackage{amsmath}
\usepackage{amssymb}
\usepackage{amsfonts}
\usepackage{epsfig}
\usepackage{srcltx}
\usepackage{subfigure}
\usepackage{cite}
\usepackage{float}
\usepackage{mathtools}
\usepackage[a4paper, mag=1000, includefoot, left=2cm, right=2cm, top=2cm, bottom=2cm, headsep=1cm, footskip=1cm]{geometry}
 \usepackage[unicode,colorlinks]{hyperref}
 \hypersetup{colorlinks=true, citecolor=blue, linkcolor=blue}

\newtheorem{Th}{Theorem}
\newtheorem{Lem}{Lemma}

\begin{document}

\thispagestyle{empty}

\title[ ]{Stability of asymptotically Hamiltonian systems with damped oscillatory and stochastic perturbations}

\author[O.A. Sultanov]{Oskar A. Sultanov}

\address{
Institute of Mathematics, Ufa Federal Research Centre, Russian Academy of Sciences, Chernyshevsky street, 112, Ufa 450008 Russia.}
\email{oasultanov@gmail.com}


\maketitle

{\small
\begin{quote}
\noindent{\bf Abstract.} 
A class of asymptotically autonomous systems on the plane with oscillatory coefficients is considered. It is assumed that the limiting system is Hamiltonian with a stable equilibrium. The effect of damped multiplicative stochastic perturbations of white noise type on the stability of the system is discussed. It is shown that different long-term asymptotic regimes for solutions are admissible in the system and the stochastic stability of the equilibrium depends on the realized regime. In particular, we show that stable phase locking is possible in the system due to decaying stochastic perturbations. The proposed analysis is based on a combination of the averaging technique and the construction of stochastic Lyapunov functions. 

 \medskip



\end{quote}
}
{\small

\section{Introduction}
The effect of perturbations on dynamical systems is a classical problem in qualitative and asymptotic theories of differential equations. This paper investigates the influence of perturbations with intensities decaying  in time on the stability of planar autonomous Hamiltonian systems. Such asymptotically autonomous systems have been studied in many papers. For example, it follows from~\cite{LM56,LDP74} that damped perturbations can preserve the asymptotic properties of solutions of the corresponding limiting autonomous systems. However, in the general case, the behaviour of perturbed and unperturbed trajectories can differ significantly~\cite{HT94,OS21IJBC}. In this case, both the properties of the unperturbed system and the class of damped perturbations are important~\cite{LRS02,KS05,MR08,OS22Non}. 

This paper focusses on the influence of stochastic perturbations. It is well known that even weak noise can destroy the stability of dynamical system and cause the exit of trajectories from any bounded domain~\cite{FW98}. The effect of stochastic disturbances on the qualitative behaviour of solutions of autonomous systems has been studied in many papers (see, for example,~\cite{RH64,NSN90,BG02,AMR08,DNR11,BRR15,TSD18}). Damped stochastic perturbations of scalar autonomous systems were discussed in~\cite{AGR09,ACR11,KT13}. Stochastic bifurcations and the long-term asymptotic behaviour of solutions of Hamiltonian systems with decaying stochastic perturbations were studied in~\cite{OS22IJBC,OS22arxiv}.

We consider a special class of damped perturbations with oscillating coefficients. It is known that oscillating perturbations can be effectively used to control the dynamics of nonlinear systems due to resonance and phase-locking phenomena~\cite{BN10,OK19}. In the deterministic case, the effect of such decaying perturbations on the long-term dynamics of autonomous systems was discussed in several papers. In particular, asymptotic analysis of linear systems was done in~\cite{PN06,BN10}. Bifurcations of equilibria and possible asymptotic modes for solutions of nonlinear systems were studied in~\cite{OS21DCDS,OS21JMS}. Damped oscillatory perturbations with chirped-frequency were considered in~\cite{OS23DCDSB}. However, taking into account a noise in the models can change the corresponding dynamics due to stochastic bifurcations (see, for instance,~\cite{OS22IJBC}). To the best of the author's knowledge the combined effect of damped oscillatory and stochastic perturbations on the stability of systems has not been thoroughly investigated. This is the aim of the present paper.

Thus, the present paper investigates the stability of Hamiltonian systems subject to multiplicative noise with oscillating coefficients. It is assumed that the intensity of perturbations fades with time and the limiting system has a neutrally stable equilibrium. Possible asymptotic regimes in the perturbed stochastic system and the conditions for stability of the equilibrium are discussed. In the analysis of stability, the method of stochastic Lyapunov functions is used~\cite{HK67,RKh12}.

The paper is organized as follows. In Section~\ref{PS}, the formulation of the problem is given and the class of damped perturbations is described. The main results are presented in Section~\ref{MR}. The proofs are contained in the subsequent sections. In particular, in Section~\ref{ChOfVar}, the transformation is constructed that simplifies the drift terms of the perturbed system in the asymptotics at infinity in time. The study of the structure of the simplified equations in Section~\ref{AR} leads to a description of possible asymptotic regimes in the system. The stability analysis based on the construction of stochastic Lyapunov functions is contained in sections~\ref{SPL} and~\ref{SPD} for different regimes. Section~\ref{SEx} considers the examples of stochastic systems and the application of the proposed theory. The paper concludes with a brief discussion of the results obtained.

\section{Problem statement}
\label{PS}
Consider the system of It\^{o} stochastic differential equations
\begin{equation}\label{FulSys}
	d {\bf x}(t) = {\bf a}({\bf x}(t),t) dt + {\bf A}({\bf x}(t),t)\,d{\bf w}(t), \quad t>t_0>0, \quad {\bf x}(t_0)={\bf x}_0\in\mathbb R^2,
\end{equation}
where ${\bf x}=(x_1,x_2)^T$, ${\bf w}(t)=(w_1(t),w_2(t))^T$ is a two dimensional Weiner process on a probability space $(\Omega,\mathcal F,\mathbb P)$, ${\bf a}({\bf x},t)=(a_1(x_1,x_2,t),a_2(x_1,x_2,t))^T$ is a vector function and ${\bf A}({\bf x},t)=\{\alpha_{i,j}(x_1,x_2,t)\}_{2\times 2}$ is a $2\times2$ matrix. The functions $a_i(x_1,x_2,t)$ and $\alpha_{i,j}(x_1,x_2,t)$, defined for all $(x_1,x_2,t)\in\mathbb R^2\times\mathbb R_+$, are infinitely differentiable and do not depend on $\omega\in\Omega$. It is also assumed that 
\begin{gather}\label{zero}
{\bf a}(0,t)\equiv 0, \quad {\bf A}(0,t)\equiv 0,
\end{gather}
and there exists $M>0$ such that ${\bf a}({\bf x},t)$ and ${\bf A}({\bf x},t)$ satisfy the Lipschitz condition:
\begin{gather}\label{lip}
\begin{split}
&|{\bf a}({\bf x}_1,t)-{\bf a}({\bf x}_2,t)|\leq M |{\bf x}_1-{\bf x}_2|, \\
& \|{\bf A}({\bf x}_1,t)-{\bf A}({\bf x}_2,t)\|\leq M |{\bf x}_1-{\bf x}_2|
\end{split}
\end{gather}
for all ${\bf x}_1, {\bf x}_2\in \mathbb R^2$ and $t\geq t_0$. Here, $|{\bf x}|=\sqrt{x_1^2+x_2^2}$ and $\|\cdot\|$ is the operator norm coordinated with the norm $|\cdot|$ of $\mathbb R^2$. Note that under these assumptions system \eqref{FulSys} has a unique continuous (with probability one) solution ${\bf x}(t)=(x_1(t),x_2(t))^T$ for all $t\geq t_0$ and for any initial point ${\bf x}_0\in\mathbb R^2$ (see, for example,~\cite[\S 5.2]{BOks98}).  

Furthermore, system \eqref{FulSys} is assumed to be asymptotically autonomous such that for each fixed ${\bf x}\in \mathbb R^2$ there exists the limits
\begin{equation*}
	\lim_{t\to\infty} {\bf a}({\bf x},t)={\bf a}_0({\bf x}), \quad \lim_{t\to\infty} {\bf A}({\bf x},t)={\bf 0}, 
\end{equation*}
where
\begin{gather}\label{H0as}
{\bf a}_0({\bf x})\equiv \begin{pmatrix} \partial_{x_2} H(x_1,x_2) \\ -\partial_{x_1} H(x_1,x_2) \end{pmatrix}, \quad 
H(x_1,x_2)=\frac{|{\bf x}|^2}{2}+\mathcal O(|{\bf x}|^3), \quad |{\bf x}|\to 0.
\end{gather}
In this case, the corresponding limiting system
\begin{equation}\label{LimSys}
	\frac{d{\bf x}}{dt}={\bf a}_0({\bf x})
\end{equation}
is Hamiltonian with a stable fixed point at the origin $(0,0)$, and it is assumed that there exist $E_0>0$ and $r_0>0$ such that for all $E\in (0,E_0]$ the level lines $\{(x_1,x_2)\in\mathbb R^2: H(x_1,x_2)=E\}$, lying in the ball $\mathcal B_0=\{(x_1,x_2)\in\mathbb R^2: |{\bf x}|\leq r_0\}$, are closed curves and correspond to periodic solutions of system \eqref{LimSys} with the period $T(E)=2\pi/\nu(E)$, $\nu(E)\neq 0$ for all $E\in (0,E_0]$. The value $E=0$ corresponds to the equilibrium ${\bf x}(t)\equiv 0$. It can easily be checked that $\nu(E)=1+\mathcal O(E)$ as $E\to 0$. We also assume that $\mathcal B_0$ does not contain any fixed points of system \eqref{LimSys}, except for the origin.

Thus, system \eqref{FulSys} can be viewed as a perturbation of the autonomous Hamiltonian system \eqref{LimSys}. Let us describe the class of damped stochastic perturbations considered in this paper. It is assumed that
\begin{equation}\label{HFBas}
	\begin{split}
		  &{\bf a}({\bf x},t)-{\bf a}_0({\bf x})\sim \sum_{k=1}^\infty t^{-\frac{k}{q}} {\bf a}_k({\bf x},S(t)), \quad 
			{\bf A}({\bf x},t)\sim \sum_{k=1}^\infty t^{-\frac{k}{q}} {\bf A}_{k}({\bf x},S(t))
	\end{split}
\end{equation}
as $t\to\infty$ uniformly for all ${\bf x}\in \mathcal B_0$ with $q\in\mathbb Z_+=\{1,2,\dots\}$, where the coefficients ${\bf a}_k({\bf x},S)$ and ${\bf A}_k({\bf x},S)=\{\alpha_{i,j,k}(x_1,x_2,S)\}_{2\times 2}$ are $2\pi$-periodic with respect to $S$,
\begin{gather}\label{Sform}
S(t)\equiv\sum_{k=0}^{q-1} s_k t^{1-\frac{k}{q}}+s_{q} \log t, \quad s_k={\hbox{\rm const}},
\end{gather}
and $s_0$ satisfies a resonance condition
\begin{gather}\label{rescond}
s_0=\varkappa \nu(0)
\end{gather}
with some $\varkappa\in\mathbb Z_+$.
Note that the series in \eqref{HFBas} are asymptotic as $t\to\infty$. In other words, it is assumed that for all $n\in \mathbb Z_+$ the following estimates hold:
\begin{gather*}
|{\bf a}({\bf x},t)-\sum_{k=0}^{n-1}t^{-\frac{k}{q}} {\bf a}_k({\bf x},S(t))|=\mathcal O(t^{-\frac{n}{q}}), \quad
\|{\bf A}({\bf x},t)-\sum_{k=1}^{n-1}t^{-\frac{k}{q}} {\bf A}_k({\bf x},S(t))\|=\mathcal O(t^{-\frac{n}{q}})
\end{gather*} 
as $t\to\infty$ uniformly for all ${\bf x}\in \mathcal B_0$. 

Note that power-law decaying perturbations arise in the context of many nonlinear and non-autonomous problems~\cite{IKNF06,BG08,KF13,Pan21,Dong22}. It is known that damped deterministic perturbations can affect the stability conditions and can lead to different asymptotic regimes~\cite{OS21DCDS,OS21IJBC,OS22Non}. The influence of such stochastic perturbations is discussed in this paper. Consider the simplest example given by the following linear system with a damped oscillatory perturbation of white noise type:
\begin{gather}\label{Ex0}
\begin{split}
&dx_1=x_2\,dt, \\
& dx_2=\left(-x_1+  t^{-1} b_0 x_2\right)\, dt + t^{-\frac{p}{2}}c_1 x_1\cos S(t)\, dw_2(t), \quad t\geq 1,
\end{split}
\end{gather}
where $S(t)\equiv s_0 t$, $p\in\mathbb Z_{+}$, $b_0,c_1,s_0={\hbox{\rm const}}$. It can easily be checked that system \eqref{Ex0} is of the form \eqref{FulSys} with $q=2$, 
\begin{gather*}
{\bf a}({\bf x},t)\equiv {\bf a}_0({\bf x})+t^{-1}
\begin{pmatrix}
0\\b_0 x_2
\end{pmatrix}
, 
\quad {\bf A}({\bf x},t)\equiv t^{-\frac{p}{2}} 
	\begin{pmatrix} 0&0\\
	0&c_1 x_1 \cos S(t)
	\end{pmatrix}, \quad 
	H(x_1,x_2)\equiv \frac{|{\bf x}|^2}{2}.
\end{gather*}
The corresponding limiting autonomous system ($b_0=c_1=0$) has $2\pi$-periodic general solution $x_1(t+\phi;E)\equiv\sqrt{2E}\cos  (t+\phi)$, $x_2(t+\phi;E)\equiv -\sqrt{2E}\sin (t+\phi)$, where $E$ and $\phi$ are arbitrary constants and $\nu(E)\equiv 1$. In the case of a deterministic damped perturbation ($b_0\neq 0$, $c_1=0$), the asymptotics of the general solution (see, for instance,~\cite{WW66})
\begin{align*}
{\bf x}(t) = 
t^{\frac{b_0}{2}} 
\begin{pmatrix}
\sqrt{2E}\cos  (t+\phi)+\mathcal O(t^{-1})\\
-\sqrt{2E}\sin  (t+\phi)+\mathcal O(t^{-1})
\end{pmatrix}
\end{align*}
as $t\to\infty$ shows that the stability of the equilibrium depends on the sign of the parameter $b_0$ (see Fig.~\ref{Fig12}, a). The numeric analysis of system \eqref{Ex0} with $b_0\neq 0$, $c_1\neq 0$ and $s_0=1$ indicates that the stability of the equilibrium in the full system depends on the degree of decay of the stochastic perturbation. In particular, if $p=2$, the stability conditions seem to be the same as in the previous case when $c_1=0$ (see Fig.~\ref{Fig12}, b). However, if $p=1$, the stability of the equilibrium $(0,0)$ changes as the parameter $b_0$ passes through a certain non-zero critical value $b_\ast$ (see Fig.~\ref{Fig12}, c). It will be shown in Section~\ref{SEx} that such shift in the stability boundary arises due to resonance capture in the stochastic system and its value depends, in particular, on the parameter $s_0$. 

\begin{figure}
\centering
\subfigure[$c_1=0$]{
\includegraphics[width=0.3\linewidth]{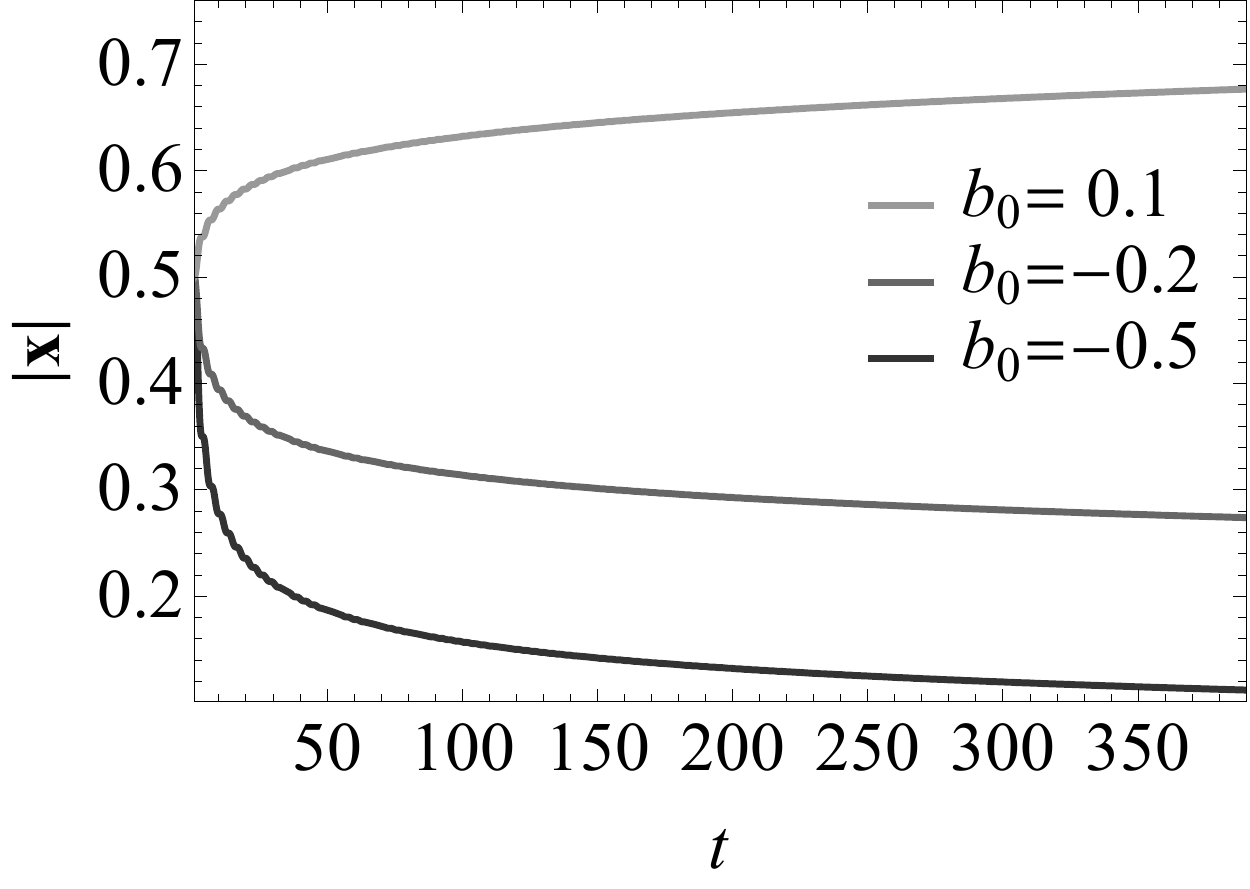}
}
\hspace{1ex}
 \subfigure[$c_1=1$, $p=2$]{
 \includegraphics[width=0.3\linewidth]{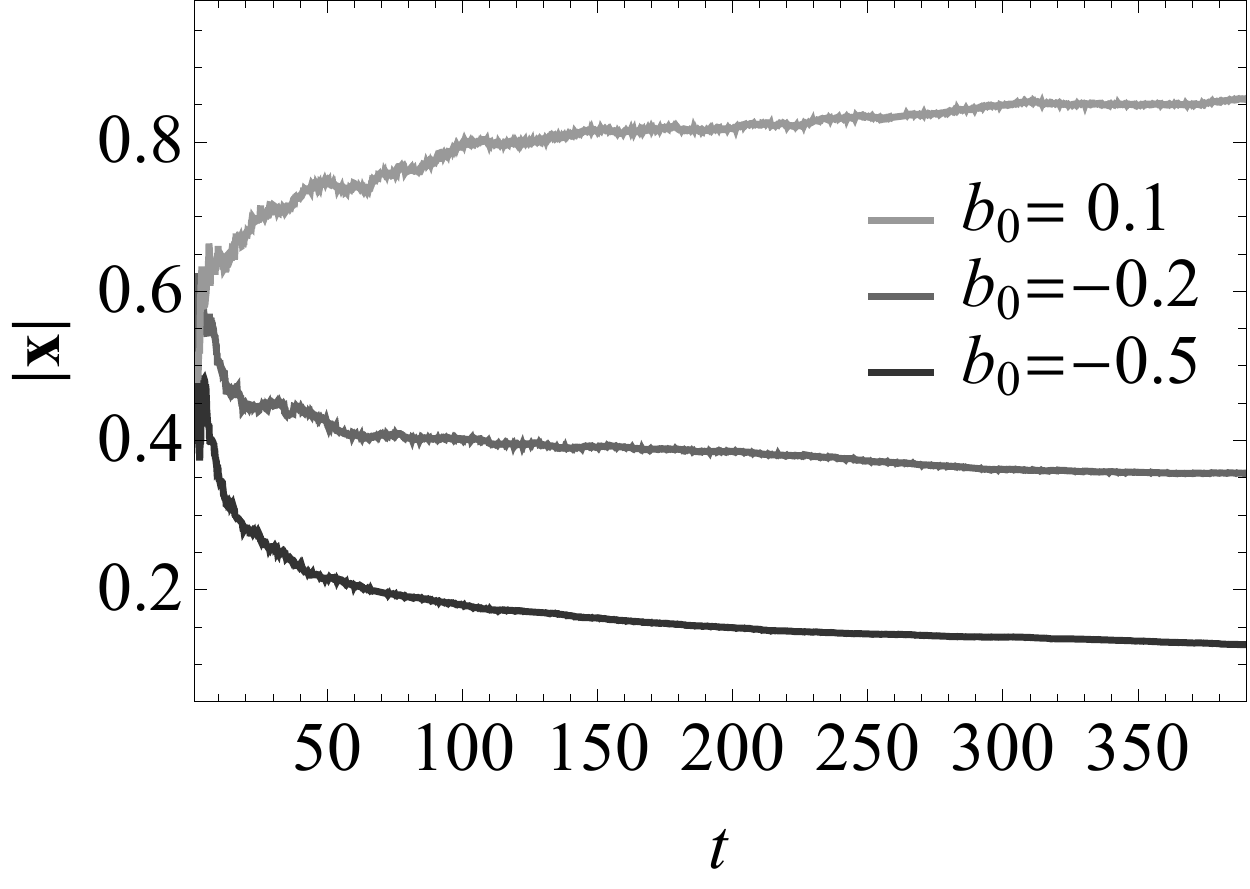}
}
\hspace{1ex}
\subfigure[$c_1=1$, $p=1$]{
 \includegraphics[width=0.3\linewidth]{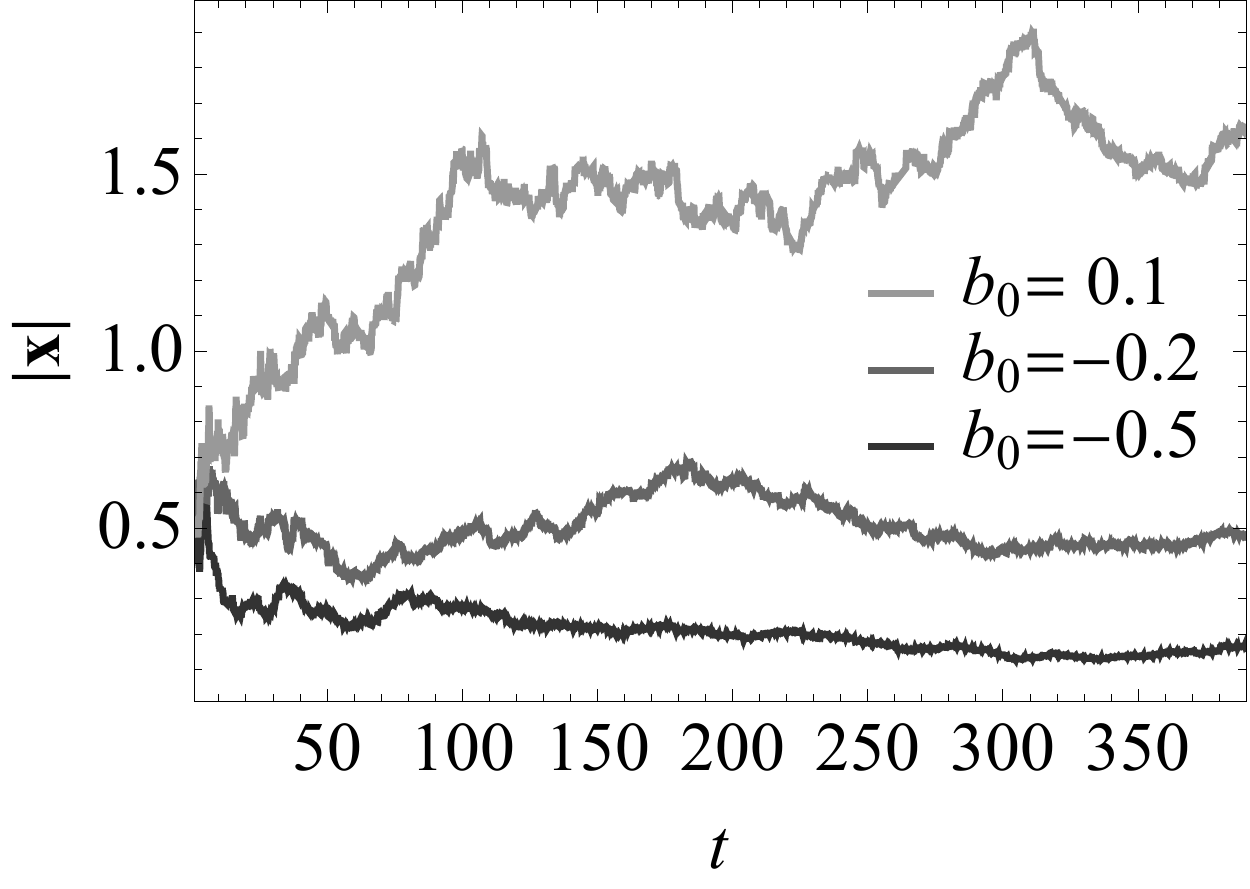}
}
\caption{\small The evolution of $|{\bf x}(t)|$ for sample paths of solutions to system \eqref{Ex0} with different values of the parameters.} \label{Fig12}
\end{figure}

\section{Main results}
\label{MR}
Let  ${\bf x}_\ast(t)\equiv (\xi_{1}(t,E), \xi_{2}(t,E))^T$ be a $T(E)$-periodic solution of the limiting system \eqref{LimSys} such that $H(\xi_{1}(t,E), \xi_{2}(t,E))\equiv E$, $\xi_1(0,E)>0$ and $\xi_2(0,E)= 0$ for all $E\in (0,E_0]$. Define the domain 
\begin{gather*}
\mathcal D(E_0):=\{(x_1,x_2)\in\mathcal B_0: H(x_1,x_2)\leq E_0\}.
\end{gather*}
Then we have the following:
\begin{Th}\label{Th1}
Let system \eqref{FulSys} satisfy assumptions \eqref{zero}, \eqref{lip}, \eqref{H0as}, \eqref{HFBas}, \eqref{Sform} and \eqref{rescond}. Then, for all $N\in \mathbb Z_+$ and $\epsilon>0$ there exist $\ell\in \mathbb Z_+\cup\{0\}$, $t_\ast\geq \max\{t_0,1\}$, $v_0>0$ and a chain of invertible transformations $(x_1,x_2)\to (E,\varphi)\to (R,\theta)\to (v,\psi)$,
\begin{gather}
\label{exch1} x_1(t)= \xi_1\left(\frac{\varphi(t)}{\nu(E(t))},E(t)\right), \quad x_2(t)= \xi_2\left(\frac{\varphi(t)}{\nu(E(t))},E(t)\right), \\
\label{Etheta} E(t)=t^{-\frac{2\ell}{q}} \left(R(t)\right)^2, \quad \varphi(t)=\varkappa^{-1}S(t)+\theta(t), \\
 \label{exch11}
v(t)=R(t)+\tilde{v}_N (R(t),\theta(t),t), \quad  \psi(t)=\theta(t)+\tilde{\psi}_N(R(t),\theta(t),t),
\end{gather}
where the functions $\tilde{v}_N(R,\theta,t)$, $\tilde{\psi}_N(R,\theta,t)$ are $2\pi$-periodic in $\theta$ and satisfy the inequalities 
\begin{gather*}
|\tilde{v}_N(R,\theta,t)|\leq \epsilon R,\quad 
|\tilde{\psi}_N(R,\theta,t)|\leq \epsilon \quad 
\forall\, R\in [0,E_0^{1/2} t_\ast^{\ell/q}], \quad  \theta\in\mathbb R, \quad  t\geq t_\ast,
\end{gather*}
such that for all $(x_1,x_2)\in \mathcal D(E_0)$ and $t\geq t_\ast$ system \eqref{FulSys} can be transformed into
\begin{gather}\label{EQs}
d {\bf z}(t) = {\bf g}_N({\bf z}(t),t) \, dt + {\bf G}_N ({\bf z}(t),t)\, d{\bf w}(t), \quad {\bf z}=(v,\psi)^T
\end{gather}
with ${\bf g}_N({\bf z},t)\equiv (\Lambda_N(v,\psi,t), \Omega_N(v,\psi,t))^T$ and
${\bf G}_N({\bf z},t)\equiv \{\sigma_{i,j}(v,\psi,t)\}_{2\times 2}$, defined for all $v\in [0,v_0]$, $\psi\in\mathbb R$ and $t\geq t_\ast$, such that
\begin{equation}\label{LO}
\begin{array}{l}
\displaystyle
\Lambda_N(v,\psi,t)\equiv\sum_{k=1}^N t^{-\frac{k}{ q}}\Lambda_k(v,\psi)+\tilde\Lambda_N(v,\psi,t), \\ 
\displaystyle \Omega_N(v,\psi,t)\equiv \sum_{k=1}^N t^{-\frac{k}{ q}}\Omega_k(v,\psi)+\tilde\Omega_N(v,\psi,t),
\end{array}
\end{equation}
where 
\begin{gather*}
\tilde\Lambda_N(v,\psi,t)=\mathcal O(t^{-\frac{N+1}{q}}), \quad  
\tilde\Omega_N(v,\psi,t)=\mathcal O(t^{-\frac{N+1}{q}}), \quad  
\sigma_{i,j}(v,\psi,t)=\mathcal O(t^{-\frac{1}{q}})
\end{gather*} 
as $t\to\infty$ uniformly for all $v\in [0,v_0]$ and $\psi\in\mathbb R$. The functions $\Lambda_N(v,\psi,t)$, $\Omega_N(v,\psi,t)$, $\sigma_{i,j}(v,\psi,t)$ are $2\pi$-periodic in $\psi$ and satisfy the following estimates:
\begin{gather*}
	\Lambda_N(v,\psi,t)=\mathcal O(v), \quad 
	\Omega_N(v,\psi,t)=\mathcal O(1), \quad 
	\sigma_{1,j}(v,\psi,t)=\mathcal O(v), \quad 
	\sigma_{2,j}(v,\psi,t)=\mathcal O(1)
\end{gather*}
as $v\to 0$ uniformly for all $\psi\in\mathbb R$ and $t\geq t_\ast$.
\end{Th} 
The proof is based on the averaging of the system with respect to $S(t)$ and is contained in Section~\ref{ChOfVar}.

Thus, the transformation described in Theorem~\ref{Th1} simplifies the first $N$ asymptotic terms of system \eqref{FulSys} as $t\to\infty$. In this case some leading terms in the transformed system may disappear due to the zero mean. Let $n\geq 1$ and $m\geq 1$ denote the integers associated with the first non-zero terms in \eqref{LO}:
\begin{equation}\label{LnOm}
\begin{array}{ lclcl}
\Lambda_i(v,\psi)\equiv 0,&\quad&  \forall\, i<n, &\quad &  \Lambda_n(v,\psi)\not\equiv 0,\\
\Omega_j(v,\psi)\equiv 0,&\quad & \forall\, j<m, &\quad& \Omega_m(v,\psi)\not\equiv 0.
\end{array}
\end{equation}

It is clear that the behaviour of solutions to system \eqref{FulSys} in the vicinity of the equilibrium ${\bf x}(t)\equiv 0$ is determined by the dynamics of transformed system \eqref{EQs} with $v(t)$ close to zero.  
With this in mind, we consider the following additional assumptions (separately) on the behaviour of the function $\Lambda_N(v,\psi,t)$ as $v\to 0$:
\begin{gather}\label{Lnas1}
 \Lambda_n(v,\psi)=v\left(\lambda_n(\psi)+\mathcal O(v)\right), \quad v\to 0;
\end{gather}
and 
\begin{gather}\label{Lnas2}
\exists\, h\geq 2,\, l\geq 1: \quad 
 \Lambda_k(v,\psi)=
\begin{cases}
	\displaystyle v^h\left(\lambda_{k,h}(\psi)+\mathcal O(v)\right), &  k<n+l,\\
	\displaystyle v\left(\lambda_{k}(\psi)+\mathcal O( v)\right), & k\geq n+l,
\end{cases}\quad v\to 0,
\end{gather}
where
$\lambda_n(\psi)$, $\lambda_{n,h}(\psi)$, $\lambda_{n+l}(\psi)$ are non-zero $2\pi$-periodic functions. It follows easily that in the case of \eqref{Lnas1} the leading asymptotic term of $\Lambda_N(v,\psi,t)$ as $t\to\infty$ has non-zero linear part with respect to $v$. In the case of \eqref{Lnas2} the leading term is strictly nonlinear.

Let us also specify the behaviour of the function $\Omega_N(v,\psi,t)$ as $v\to 0$. We assume that
\begin{gather}\label{Omas}
 \Omega_m(v,\psi)=
 	\omega_{m,0}(\psi)+\omega_{m,1}(\psi)v+\mathcal O(v^2),   \quad v\to 0,
\end{gather}
where $\omega_{m,0}(\psi)$ is a $2\pi$-periodic function. Note that by appropriate choosing $\ell\in\mathbb Z_{+}\cup\{0\}$ in \eqref{Etheta}, we can assume without loss of generality that $\omega_{m,0}(\psi)\not\equiv 0$. 
Moreover, consider the following two cases:
\begin{gather}
\label{zerom} \exists\, \phi_0\in\mathbb R: \quad \omega_{m,0}(\phi_0)=0, \quad \vartheta_{m}:=\omega_{m,0}'(\phi_0)\neq 0;\\
\label{nzerom} \omega_{m,0}(\psi)\neq 0 \quad \forall \, \psi\in\mathbb R.
\end{gather}
Each of these cases corresponds to a specific asymptotic regime associated with the long-term behaviour of the phase of solutions to system \eqref{FulSys}. Indeed, consider the axillary system of ordinary differential equations, which is obtained from system \eqref{EQs} by dropping the stochastic part
\begin{gather}\label{DetSys}
\frac{du}{dt}=\Lambda_N(u,\phi,t), \quad \frac{d\phi}{dt}=\Omega_N(u,\phi,t), \quad t\geq t_\ast
\end{gather}
with $N= \max\{n,m\}$. We have the following:
\begin{Lem}\label{Lem1}
Let assumptions \eqref{LnOm}, \eqref{Omas} and \eqref{zerom} hold with $m\leq q$. If $\vartheta_m<0$, then for all $\varepsilon>0$ system \eqref{DetSys} has a particular solution $u_\varepsilon(t)\equiv 0$, $\phi_\varepsilon(t)\equiv \phi_0+\varepsilon\phi_1(t)$ such that 
\begin{gather}\label{zetaas}
\phi_1(t)=\begin{cases} \mathcal O(1), & m<q,\\
\mathcal O(t^{-c}), & m=q,
\end{cases}
\end{gather}
as $t\to\infty$ with some $c>0$.
\end{Lem}

\begin{Lem}\label{Lem2}
Let assumptions \eqref{LnOm}, \eqref{Omas} and \eqref{nzerom} hold with $m\leq q$. Then system \eqref{DetSys} has a one-parameter family of solutions such that $u(t)\equiv 0$ and $|\phi(t)|\to\infty$ as $t\to\infty$.
\end{Lem}

The justification of Lemmas~\ref{Lem1} and~\ref{Lem2} is contained in Section~\ref{AR}.

First, consider the case of phase locking. 
Let 
\begin{gather} \label{HPhi}
	E=H(x_1,x_2),  \quad \varphi=\Phi(x_1,x_2) 
\end{gather}
for all $(x_1,x_2)\in\mathcal D(E_0)$ denote the inverse transformation to \eqref{exch1}.
Define the function 
\begin{gather*}
d({\bf x},t)\equiv \sqrt{t^{\frac{2\ell}{q}}H(x_1,x_2)+|\Phi(x_1,x_2)-\varkappa^{-1}S(t)-\phi_\varepsilon(t)|^2},
\end{gather*} 
where $\phi_\varepsilon(t)$ is the function defined in Lemma~\ref{Lem1}, and introduce the following additional assumption on the class of perturbations:
\begin{gather}
	\label{classPert}
		\exists\, \mu>0: \quad \left|{\hbox{\rm tr}}({\bf A}^T({\bf x},t){\bf A}({\bf x},t))\right|\leq \mu^2 t^{-\frac{2p}{q}}|{\bf x}|^2
\end{gather}
for all ${\bf x}\in \mathcal B_0 $ and $t\geq t_0$ with some $p\in \mathbb Z_{+}$.
Then we have
\begin{Th}\label{Th2}
Let system \eqref{FulSys} satisfy \eqref{zero}, \eqref{lip}, \eqref{H0as}, \eqref{HFBas}, \eqref{Sform}, \eqref{rescond} and assumptions \eqref{LnOm}, \eqref{Lnas1}, \eqref{Omas}, \eqref{zerom}, \eqref{classPert} hold  with $1\leq m\leq q$ and $\vartheta_m<0$. 
 If $\lambda_{n}(\phi_0)<0$, then for all $\varepsilon_1>0$ and $\varepsilon_2>0$ there exist $\delta>0$ and $t_s>0$ such that the solution ${\bf x}(t)$ of system \eqref{FulSys}
with initial data ${\bf x}(t_s)={\bf x}_0$, $d({\bf x}_0,t_s)< \delta$ satisfies
\begin{gather}\label{def1}
\mathbb P\left(\sup_{ 0\leq t-t_s\leq \mathcal T} \left\{t^{-\varsigma} |d({\bf x}(t),t)|\right\}>\varepsilon_1\right)<\varepsilon_2,
\end{gather}
where 
\begin{gather}\label{mathcalT}
\varsigma=
\begin{cases}
0, & n\leq m,\\
\frac{n-m}{2q}, & n> m,
\end{cases} \qquad 
\mathcal T=\begin{cases} 
C_0 \delta^2 \mu^{-2}, & 2p<q,\\
t_s (e^{C_0\delta^2 \mu^{-2}}-1), & 2p=q,\\
\infty, & 2p>q
\end{cases}
\end{gather}
 with some $C_0={\hbox{\rm const}}>0$.
\end{Th}

We see that if $n>m$, Theorem~\ref{Th2} provides only ``weak'' stochastic stability with a decaying weight $t^{-\varsigma}$. Note also that the estimate \eqref{def1} corresponds to the stability in probability~\cite[\S 5.3]{RKh12} for the equilibrium ${\bf x}(t)\equiv 0$ and some phase dynamics. 

Let us remark that  a non-trivial solution with amplitude close to zero may appear in the strictly non-linear case \eqref{Lnas2}. This effect is due to the Hopf bifurcation in some scaled variables (see, for instance,~\cite{OS22Non}). In particular, we have the following: 
\begin{Lem}\label{Lem3}
Let assumptions \eqref{LnOm}, \eqref{Omas} and \eqref{zerom} hold with $n+l=m\leq q$. If $\lambda_{n,h}(\phi_0)<0$, $\lambda_{n+l}(\phi_0)>0$ and $\vartheta_m<0$, then for all $\varepsilon>0$ system \eqref{DetSys} has a particular solution 
\begin{gather*}
u_\varepsilon(t)\equiv t^{-\kappa} (u_0+\varepsilon u_1(t)), \quad  
\phi_\varepsilon(t)\equiv \phi_0+\varepsilon\phi_1(t)
\end{gather*} 
with $u_1(t)=\mathcal O(1)$ and $\phi_1(t)=\mathcal O(1)$ as $t\to\infty$,  where
\begin{gather}\label{kappau0}
\kappa=\frac{l}{(h-1)q}, \quad u_0=\left(\frac{\lambda_{n+l}(\phi_0)+\delta_{n+l,q}\kappa}{|\lambda_{n,h}(\phi_0)|}\right)^{\frac{1}{h-1}}.
\end{gather}
\end{Lem}

Define the function 
\begin{gather*}
\tilde d({\bf x},t)\equiv \sqrt{t^{2\kappa}\left|t^{\frac{\ell}{q}}\sqrt {H(x_1,x_2)}-u_\varepsilon(t)\right|^2+\left|\Phi(x_1,x_2)-\varkappa^{-1}S(t)-\phi_\varepsilon(t)\right|^2},
\end{gather*} 
where $u_\varepsilon(t),\phi_\varepsilon(t)$ is the solution of system \eqref{DetSys} described in Lemma~\ref{Lem3}.
Then we have
\begin{Th}\label{Th3}
Let system \eqref{FulSys} satisfy \eqref{zero}, \eqref{lip}, \eqref{H0as}, \eqref{HFBas}, \eqref{Sform}, \eqref{rescond} and assumptions \eqref{LnOm}, \eqref{Lnas2}, \eqref{Omas}, \eqref{zerom}, \eqref{classPert} hold with $1\leq m\leq q$ and $\vartheta_m<0$. 
\begin{itemize}
	\item  If $\lambda_{n,h}(\phi_0)<0$, $\lambda_{n+l}(\phi_0)<0$, then for all $\varepsilon_1>0$ and $\varepsilon_2>0$ there exist $\delta>0$ and $t_s>0$ such that the solution ${\bf x}(t)$ of system \eqref{FulSys}
with initial data ${\bf x}(t_s)={\bf x}_0$, $d({\bf x}_0,t_s)< \delta$ satisfies \eqref{def1} 
with $\mathcal T$, defined by \eqref{mathcalT}, and 
\begin{gather*}
\varsigma=
\begin{cases}
0, & n+l\leq m,\\
\frac{n+l-m}{2q}, & n+l>m.
\end{cases}
\end{gather*}
\item  If $\lambda_{n,h}(\phi_0)<0$, $\lambda_{n+l}(\phi_0)>0$ and $n+l=m$, then for all $\varepsilon_1>0$ and $\varepsilon_2>0$ there exist $\delta>0$ and $t_s>0$ such that the solution ${\bf x}(t)$ of system \eqref{FulSys}
with initial data ${\bf x}(t_s)={\bf x}_0$, $\tilde d({\bf x}_0,t_s)<\delta$ satisfies
\begin{gather}\label{deftilde}
\mathbb P\left(\sup_{ 0\leq t-t_s\leq \mathcal T}   |\tilde d({\bf x}(t),t)| >\varepsilon_1\right)<\varepsilon_2
\end{gather}
with $\mathcal T$, defined by \eqref{mathcalT}.
\end{itemize}
\end{Th}
The justification of Theorems~\ref{Th2},~\ref{Th3} and Lemma~\ref{Lem3} is discussed in Section~\ref{SPL}.

Thus, in the case of the phase locking, the stability of the equilibrium ${\bf x}(t)\equiv 0$ depends on the parameters $\lambda_k(\psi)$, $\lambda_{k,h}(\psi)$ of the amplitude equation calculated at the value of the phase shift $\phi_0$. In this case, the phase of the system is synchronizing with the perturbation. 

Now consider the case of assumption \eqref{nzerom}, when the phase drifting regime occurs. Let us show that in this case the equilibrium can also be stochastically stable, but under stronger conditions. We have the following:
 
\begin{Th}\label{Th4}
Let system \eqref{FulSys} satisfy \eqref{zero}, \eqref{lip}, \eqref{H0as}, \eqref{HFBas}, \eqref{Sform}, \eqref{rescond} and assumptions \eqref{LnOm}, \eqref{Lnas1}, \eqref{Omas}, \eqref{nzerom} hold  with $1\leq m\leq q$. 
 If $\lambda_{n}(\psi)<0$ for all $\psi\in \mathbb R$, then there exists $t_s\geq t_0$ such that for all $\varepsilon_1>0$ and $\varepsilon_2>0$ there exists $\delta>0$: the solution ${\bf x}(t)$ of system \eqref{FulSys}
with initial data ${\bf x}(t_s)={\bf x}_0$, $|{\bf x}_0|<\delta$ satisfies
\begin{gather}\label{defTh4}
\mathbb P\left(\sup_{ t\geq t_s} \left\{t^{\frac{\ell}{q}} |{\bf x}(t)|\right\}>\varepsilon_1\right)<\varepsilon_2.
\end{gather}
\end{Th}

We have a similar statement in the case of a strictly non-linear leading term in $\Lambda_N(v,\psi,t)$ with respect to $v$.
\begin{Th}\label{Th5}
Let system \eqref{FulSys} satisfy \eqref{zero}, \eqref{lip}, \eqref{H0as}, \eqref{HFBas}, \eqref{Sform}, \eqref{rescond} and assumptions \eqref{LnOm}, \eqref{Lnas2}, \eqref{Omas}, \eqref{nzerom} hold  with $1\leq m\leq q$. 
 If $\lambda_{n,h}(\psi)<0$ and $\lambda_{n+l}(\psi)<0$ for all $\psi\in \mathbb R$, then there exists $t_s\geq t_0$ such that for all $\varepsilon_1>0$ and $\varepsilon_2>0$ there exists $\delta>0$: the solution ${\bf x}(t)$ of system \eqref{FulSys}
with initial data ${\bf x}(t_s)={\bf x}_0$, $|{\bf x}_0|< \delta$ satisfies \eqref{defTh4}.
\end{Th}
The proofs of Theorems~\ref{Th4} and~\ref{Th5} are contained in Section~\ref{SPD}.

Thus, in this case, stability is justified only for global estimates on the coefficients $\lambda_k(\psi)$ and $\lambda_{k,h}(\psi)$. Note that such stability conditions in the phase drifting regime appear for deterministic systems without noise~\cite{OS21DCDS}.  

\section{Change of variables} \label{ChOfVar}

In this section, we construct a chain of changes of variables that transform system \eqref{FulSys} into the form \eqref{EQs}. 

\subsection{Energy-angle variables}
First, define the functions
\begin{gather*}
X_1(E,\varphi)\equiv  \xi_1\left(\frac{\varphi}{\nu(E)},E\right), \quad  X_2(E,\varphi)\equiv  \xi_2\left(\frac{\varphi}{\nu(E)},E\right).
\end{gather*} 
It can be easily checked that these functions are $2\pi$-periodic with respect to $\varphi$ and satisfy the system
\begin{gather}\label{LimSys2pi}
    \nu(E)\frac{\partial X_1}{\partial \varphi}=\partial_{X_2} H(X_1,X_2), \quad
    \nu(E)\frac{\partial X_2}{\partial \varphi}=-\partial_{X_1} H(X_1,X_2),
\end{gather}
Moreover, from the definition of the functions $\xi_1(\varphi,E)$ and $\xi_2(\varphi,E)$ it follows that
\begin{gather}
	\label{Hequiv}	H(X_1(E,\varphi),X_2(E,\varphi))\equiv E.
\end{gather}
By using $X_1(E,\varphi)$ and $X_2(E,\varphi)$, system \eqref{FulSys} can be rewriting in the energy-angle variables $(E,\varphi)$. 
Differentiating the identity \eqref{Hequiv} with respect to $E$ and using \eqref{LimSys2pi} yield
\begin{gather*}
\det\frac{\partial(X_1,X_2)}{\partial (E,\varphi)}=\begin{vmatrix}
        \partial_E X_1 & \partial_\varphi X_1\\
        \partial_E X_2& \partial_\varphi X_2
    \end{vmatrix} \equiv  \frac{1}{\nu(E)}\neq 0, \quad E\in [0,E_0].
\end{gather*}
Hence, the transformation \eqref{exch1} is invertible for all $E\in [0,E_0]$ and $\varphi\in [0,2\pi)$.
Define the operators
\begin{gather*}
\mathcal L({\bf x},{\bf a}({\bf x},t),{\bf A}({\bf x},t)) U :=  \partial_t U  + \left(\nabla_{\bf x} U\right)^T {\bf a}({\bf x},t)+\frac{1}{2}{\hbox{\rm tr}}\left({\bf A}^T({\bf x},t) {\bf H}_{\bf x}(U){\bf A}({\bf x},t)\right), \\ 
\nonumber \nabla_{\bf x} U := \begin{pmatrix}\partial_{x_1} U  \\  \partial_{x_2} U \end{pmatrix}, \quad 
{\bf H}_{\bf x}(U):= 	
\begin{pmatrix}
		\partial_{x_1}^2 U & \partial_{x_1}\partial_{x_2} U \\
		\partial_{x_2}\partial_{x_1} U & \partial_{x_2}^2 U
\end{pmatrix}
\end{gather*}
for any smooth function $U({\bf x},t)$. Recall that the inverse transformation to \eqref{exch1} is given by \eqref{HPhi}. Then, by applying It\^{o}'s formula, it can be shown that in the new variables ${\bf e}=(E,\varphi)^T$ system \eqref{FulSys} takes the form
\begin{equation*}
	d{\bf e}(t)={\bf b}({\bf e}(t),t) dt +{\bf B}({\bf e}(t),t)\,d{\bf w}(t)
\end{equation*}
with  ${\bf b}( {\bf e},t)\equiv (b_1(E,\varphi,t), b_2(E,\varphi,t))^T$ and ${\bf B}(  {\bf e},t)\equiv \{\beta_{i,j}(E,\varphi,t)\}_{2\times 2}$, where
\begin{gather*}
\begin{split} 
b_1  \equiv & \mathcal L({\bf x},{\bf a}({\bf x},t),{\bf A}({\bf x},t)) H(x_1,x_2)\Big|_{x_1=X_1(E,\varphi), x_2=X_2(E,\varphi)}, \\ 
b_2   \equiv & \mathcal L({\bf x},{\bf a}({\bf x},t),{\bf A}({\bf x},t)) \Phi(x_1,x_2)\Big|_{x_1=X_1(E,\varphi), x_2=X_2(E,\varphi)},\\
\left(\beta_{1,1}, \beta_{1,2}\right) \equiv & \left(\nabla_{\bf x} H(x_1,x_2)\right)^T {\bf A}({\bf x},t)\Big|_{x_1=X_1(E,\varphi), x_2=X_2(E,\varphi)}, \\ 
\left(\beta_{2,1}, \beta_{2,2}\right) \equiv & \left(\nabla_{\bf x} \Phi(x_1,x_2)\right)^T {\bf A}({\bf x},t)\Big|_{x_1=X_1(E,\varphi), x_2=X_2(E,\varphi)}.
\end{split}
\end{gather*}
It follows from \eqref{HFBas} that
\begin{gather}\label{bBas}
\begin{split}
 {\bf b}( {\bf e},t)\sim &\begin{pmatrix} 0\\ \nu(E)\end{pmatrix}+\sum_{k=1}^\infty t^{-\frac{k}{q}} \begin{pmatrix} b_{1,k}(E,\varphi,S(t))\\ b_{2,k}(E,\varphi,S(t))\end{pmatrix}, \\ 
 \beta_{i,j}(E,\varphi,t)\sim &\sum_{k=1}^\infty t^{-\frac{k}{q}} \beta_{i,j,k}(E,\varphi,S(t)) 
\end{split}
\end{gather}
as $t\to\infty$, where
\begin{gather}
\label{b0bk}
\begin{split} 
  b_{1,k}\equiv& 
	\displaystyle \left(\nabla_{\bf x} H\right)^T {\bf a}_k+\frac{1}{2}\sum_{i+j=k}{\hbox{\rm tr}}\left({\bf A}_i^T {\bf H}_{\bf x}(H){\bf A}_j\right)\Big|_{x_1=X_1(E,\varphi), x_2=X_2(E,\varphi)}, \\ 
	 b_{2,k}\equiv& 
	\displaystyle \left(\nabla_{\bf x} \Phi\right)^T {\bf a}_k+\frac{1}{2}\sum_{i+j=k}{\hbox{\rm tr}}\left({\bf A}_i^T {\bf H}_{\bf x}(\Phi){\bf A}_j\right)\Big|_{x_1=X_1(E,\varphi), x_2=X_2(E,\varphi)}, \\
	\left(\beta_{1,1,k}, \beta_{1,2,k}\right)\equiv & \left(\nabla_{\bf x} H\right)^T {\bf A}_k\Big|_{x_1=X_1(E,\varphi), x_2=X_2(E,\varphi)},\\
\left(\beta_{2,1,k}, \beta_{2,2,k}\right)\equiv & \left(\nabla_{\bf x} \Phi\right)^T {\bf A}_k\Big|_{x_1=X_1(E,\varphi), x_2=X_2(E,\varphi)}.
\end{split}
\end{gather}
Note that $b_{i,k}(E,\varphi,S)$ and $\beta_{i,j,k}(E,\varphi,S)$ are $2\pi$-periodic with respect to $\varphi$ and $S$. Moreover, the following asymptotic expansions hold:
\begin{equation}\label{expzero}
\begin{array}{rclcrcl}
	b_{1,k}(E,\varphi,S) &=& \displaystyle \sum_{l=2}^\infty E^{\frac{l}{2}}b_{1,k,l}(\varphi,S), &\quad &  
	b_{2,k}(E,\varphi,S) &= &\displaystyle \sum_{l=0}^\infty E^{\frac{l}{2}}b_{2,k,l}(\varphi,S), \\ 
	\beta_{1,j,k}(E,\varphi,S) &=& \displaystyle \sum_{l=2}^\infty E^{\frac{l}{2}}\beta_{1,j,k,l}(\varphi,S), &\quad & 
	\beta_{2,j,k}(E,\varphi,S) &= & \displaystyle \sum_{l=0}^\infty E^{\frac{l}{2}}\beta_{2,j,k,l}(\varphi,S)
\end{array}
\end{equation}
as $E\to 0$ uniformly for all $(\varphi,S)\in\mathbb R^2$.  
Indeed, it can easily be checked that $\nabla_{\bf x} H=\nu  (-\partial_\varphi X_2,\partial_\varphi X_1)^T$, $\nabla_{\bf x} \Phi= \nu  (\partial_E X_2,-\partial_E X_1)^T$, and 
\begin{equation}\label{deriv}
\begin{array}{l}
  \partial^2_{x_1} \begin{pmatrix} H \\  \Phi \end{pmatrix}  \equiv \nu 
	\begin{vmatrix} 
		\partial_\varphi {X_2} & \partial_\varphi \\
		\partial_E {X_2} & \partial_E
	\end{vmatrix} \begin{pmatrix} \nu \partial_\varphi X_2 \\  -\nu\partial_E X_2 \end{pmatrix}, \quad
	 \partial_{x_1}\partial_{x_2} \begin{pmatrix} H \\  \Phi \end{pmatrix}  \equiv - \nu \begin{vmatrix} 
		\partial_\varphi X_1 & \partial_\varphi \\
		\partial_E X_1 & \partial_E
	\end{vmatrix} \begin{pmatrix} \nu \partial_\varphi X_2 \\  -\nu\partial_E X_2 \end{pmatrix},\\ \\
	\partial^2_{x_2} \begin{pmatrix} H \\  \Phi \end{pmatrix}  \equiv \nu 
	\begin{vmatrix} 
		\partial_\varphi X_1 & \partial_\varphi \\
		\partial_E X_1 & \partial_E
	\end{vmatrix} \begin{pmatrix} \nu \partial_\varphi X_1 \\  -\nu\partial_E X_2 \end{pmatrix}.
\end{array}
\end{equation}
From \eqref{H0as} and \eqref{LimSys2pi} it follows that $X_1(E,\varphi)=\sqrt{2E}\cos \varphi+\mathcal O(E)$ and $X_2(E,\varphi)=-\sqrt{2E}\sin\varphi+\mathcal O(E)$ as $E\to 0$ uniformly for all $\varphi\in\mathbb R$. Hence,
\begin{gather*}
{\bf H}_{\bf x}(H)=\begin{pmatrix} 1 & 0 \\ 0 & 1 \end{pmatrix}+\mathcal O(E^{\frac{1}{2}}), \quad 
{\bf H}_{\bf x}(\Phi)=\frac{1}{2E}\begin{pmatrix} \sin 2\varphi  & \cos 2\varphi  \\ \cos 2\varphi  & -\sin 2\varphi    \end{pmatrix}+\mathcal O(E^{-\frac 12}), \quad E\to 0.
\end{gather*}
Combining this with \eqref{b0bk} and \eqref{deriv}, we obtain \eqref{expzero}.

\subsection{Amplitude and phase difference}
To describe effects associated with the influence of oscillating perturbations in \eqref{FulSys}, we introduce a scaled amplitude and a phase difference variable in the form \eqref{Etheta},
where $\ell$ is some non-negative integer. We take $\ell=0$ if $\nu(E)\equiv {\hbox{\rm const}}$ and $\ell>0$ if $\nu(E)\not\equiv{\hbox{\rm const}}$.  
In the new variables ${\bf y}=(R,\theta)^T$, the perturbed system is
\begin{equation}
	\label{FulSys3}
	d{\bf y}(t) =   {\bf f}({\bf y}(t),t) dt +   {\bf F}({\bf y}(t),t)\,d{\bf w}(t)
\end{equation}
with $ {\bf f}({\bf y},t)\equiv (f_1(R,\theta,t)+t^{-1}\ell R/q,f_2(R,\theta,t))^T$ and ${\bf F}({\bf y},t)\equiv \{\gamma_{i,j}(R,\theta,t)\}_{2\times 2}$, where
\begin{equation*}
\begin{array}{rcl}
 f_1(R,\theta,t)&\equiv& \displaystyle
  (2R)^{-1}\left(t^{\frac{2\ell}{q}} b_1\left(t^{-\frac{2\ell}{q}} R^2,\varkappa^{-1}S(t)+\theta,t\right)-\sum_{j=1}^2 \gamma_{1,j}^2(R,\theta,t)\right), \\
 f_2(R,\theta,t)&\equiv& \displaystyle b_2\left(t^{-\frac{2\ell}{q}} R^2,\varkappa^{-1}S(t)+\theta,t\right)-\varkappa^{-1}S'(t),\\
 \gamma_{1,j}(R,\theta,t)&\equiv& \displaystyle  (2R)^{-1} t^{\frac{2\ell}{q}}\beta_{1,j}\left(t^{-\frac{2\ell}{q}} R^2,\varkappa^{-1}S(t)+\theta,t\right),\\
 \gamma_{2,j}(R,\theta,t)&\equiv& \displaystyle  \beta_{2,j}\left(t^{-\frac{2\ell}{q}} R^2,\varkappa^{-1}S(t)+\theta,t\right).
\end{array}
\end{equation*}
From \eqref{bBas} and \eqref{expzero} it follows that	
\begin{gather*}
 {\bf f}({\bf y},t)\sim \sum_{k=1}^\infty t^{-\frac{k}{ q}} {\bf f}_k( {\bf y},S(t)), \quad 
 {\bf F}({\bf y},t)\sim \sum_{k=1}^\infty t^{-\frac{k}{ q}} {\bf F}_{k}( {\bf y},S(t)),  
\end{gather*}
as $t\to\infty$, where 
${\bf f}_k({\bf y},S)\equiv (f_{1,k}(R,\theta,S)+\delta_{k,q}\ell R/q,f_{2,k}(R,\theta,S))^T$ and ${\bf F}_{k}({\bf y},S)\equiv \{\gamma_{i,j,k}(R,\theta,S)\}_{2\times 2}$ are $2\pi$-periodic in $\theta$ and $2\pi\varkappa$-periodic in $S$. In particular, 
if $\ell=1$, we have
\begin{equation*}
\begin{array}{rcl}
 f_{1,k} &\equiv & \displaystyle  \frac{1}{2} \sum_{\substack{l+m=k\\ l\geq 0, \, m\geq 1}}b_{1,m,l+2}(\theta+\varkappa^{-1} S,S)R^{1+l}-\frac{1}{2R}\sum_{j=1}^2\sum_{\substack{l+m=k\\ l\geq 1, \, m\geq 1}}\gamma_{1,j,m}(R,\theta,S)\gamma_{1,j,l}(R,\theta,S),\\
 f_{2,k} &\equiv&\displaystyle \frac{\nu^{ ( {k}/{2})}(0)}{(k/2)!} R^{k}-\varkappa^{-1}\Big(1-\frac{k}{q}+\delta_{k,q}\Big)s_{k} +\sum_{\substack{l+m=k\\ l\geq 0,\, m\geq 1}} b_{2,m,l}(\theta+\varkappa^{-1} S,S) R^{l},\\
 \gamma_{1,j,k} &\equiv& \displaystyle  \frac{1}{2}\sum_{\substack{l+m=k\\ l\geq 0, \, m\geq 1}}  \beta_{1,j,m,l+2}(\theta+\varkappa^{-1} S,S)R^{1+l},\\
 \gamma_{2,j,k} &\equiv &\displaystyle \sum_{\substack{l+m=k\\ l\geq 0, \, m\geq 1}}  \beta_{2,j,m,l}(\theta+\varkappa^{-1} S,S)R^{l},
\end{array}
\end{equation*}
where $\delta_{k,q}$ is the Kronecker delta. Here and below we set $s_j=0$ for $j>q$, $\nu^{(k/2)}(0)=0$ for odd $k$. Thus,
$f_1(R,\theta,t)=\mathcal O(t^{-{1}/{q}})\mathcal O(R)$, $f_2(R,\theta,t)=\mathcal O(t^{-{1}/{q}})$, $\gamma_{1,j}(R,\theta,t)=\mathcal O(t^{-{1}/{q}})\mathcal O(R)$, $\gamma_{2,j}(R,\theta,t)=\mathcal O(t^{-{1}/{q}})$ as $R\to 0$ and $t\to\infty$ uniformly for all $\theta\in\mathbb R$.

\subsection{Averaging}
Note that system \eqref{FulSys3} is asymptotically autonomous with the corresponding trivial limiting system $\dot R=0$, $\dot \theta=0$. Hence, $S(t)$ changes faster in comparison to possible variations of $R(t)$ and $\theta(t)$ as $t\to\infty$. Next, we consider the change of variables averaging the drift terms of system \eqref{FulSys3}. Note that such trick is effectively used in perturbation theory~\cite{BM61,AN84,AKN06}. 
The transformation is sought in the form 
\begin{gather}\label{VnPsiN}
V_N(R,\theta,t)=R+\sum_{k=1}^N t^{-\frac{k}{ q}} v_k (R,\theta,S(t)), \quad 
\Psi_N(R,\theta,t)=\theta+\sum_{k=1}^N t^{-\frac{k}{ q}} \psi_k(R,\theta,S(t)),
\end{gather}
with some integer $N\geq 1$. The coefficients $v_k(R,\theta,S)$ and $\psi_k(R,\theta,S)$ are chosen such that in the new variables $v(t)\equiv V_N(R(t),\theta(t),t)$ and $\psi(t)\equiv \Psi_N(R(t),\theta(t),t)$ the perturbed system  takes the form \eqref{EQs}, where the functions $\Lambda_k(v,\psi)$ and $\Omega_k(v,\psi)$ are independent of $S(t)$ and the remainders $\tilde\Lambda_N(v,\psi,t)$, $\tilde\Omega_N(v,\psi,t)$ decay sufficiently fast as $t\to\infty$. By applying It\^{o}'s formula to \eqref{VnPsiN}, we obtain
\begin{equation}\label{Itovpsi}
\begin{array}{rcl}
dv&=&\displaystyle \mathcal L({\bf y},{\bf f}({\bf y},t),{\bf F}({\bf y},t))  V_N(R,\theta,t)\, dt + (\nabla_{ {\bf y}} V_N(R,\theta,t))^T {\bf F}( {\bf y},t)\, d{\bf w}, \\  
d\psi&=&\displaystyle \mathcal L({\bf y},{\bf f}({\bf y},t),{\bf F}({\bf y},t)) \Psi_N(R,\theta,t)\, dt + (\nabla_{\bf y} \Psi_N(R,\theta,t))^T {\bf F}({\bf y},t)\, d{\bf w}
\end{array}
\end{equation}
Note that the drift terms in \eqref{Itovpsi} have the following asymptotic expansions as $t\to\infty$:
\begin{align*}
&\mathcal L({\bf y},{\bf f}({\bf y},t),{\bf F}({\bf y},t)) \begin{pmatrix} V_N(R,\theta,t) \\ \Psi_N(R,\theta,t) \end{pmatrix}\sim\\
&\sum_{k=1}^\infty t^{-\frac{k}{ q}}  
\left\{
  {\bf f}_k({\bf y},S(t))
 + s_0 \partial_S \begin{pmatrix} v_k  \\ \psi_k  \end{pmatrix}   
\right\}
 \\
&+\sum_{k=2}^\infty t^{-\frac{k}{ q}} \sum_{j=1}^{k-1} \left\{\left(f_{1,j}+\delta_{j,q}\frac{2\ell R}{q}\right)\partial_{R}+f_{2,j}\partial_\theta +
 s_{j  }\left(1-\frac{j}{ q}+ \delta_{j, q}\right) \partial_S -\delta_{j,q}\frac{k-q}{ q}\right\}\begin{pmatrix} v_{k-j} \\ \psi_{k-j}  \end{pmatrix} 
\\
&+ \frac{1}{2}\sum_{k=3}^\infty t^{-\frac{k}{ q}} \sum_{i+j+l=k} 
\begin{pmatrix}{\hbox{\rm tr}}\left({\bf F}^T_{i}({\bf y},S(t)) {\bf H}_{\bf y}(v_{j} ){\bf F}_{l}({\bf y},S(t))\right) \\ {\hbox{\rm tr}}\left({\bf F}^T_{i}({\bf y},S(t)) {\bf H}_{\bf y}(\psi_{j} ){\bf F}_{l}({\bf y},S(t))\right) \end{pmatrix},
\end{align*}
where it is assumed that ${\bf f}_k({\bf y},S)\equiv0$ and ${\bf F}_{k}({\bf y},S) \equiv0 $ if $k<1$, and $v_j(R,\theta,S)\equiv \psi_j(R,\theta,S)\equiv  0$ if $j<1$ or $j>N$.  
Comparing the asymptotics of the drift terms in \eqref{Itovpsi} with \eqref{LO} gives the following chain of differential equations for determining the coefficients $v_k(R,\theta,S)$, $\psi_k(R,\theta,S)$:
\begin{gather}\label{RSys}
   s_0 \partial_S \begin{pmatrix} v_k \\ \psi_k \end{pmatrix} =
\begin{pmatrix}
        \Lambda_k(R,\theta)\\
        \Omega_k(R,\theta)
\end{pmatrix} 
-{\bf f}_k( {\bf y},S)+ \tilde {\bf f}_k({\bf y},S),
\quad k\geq 1,
\end{gather}
where the additional terms $\tilde {\bf f}_k({\bf y},S)\equiv (\tilde{f}_{1,k}(R,\theta,S), \tilde{f}_{2,k}(R,\theta,S))^T$ in the right-hand side are expressed explicitly through $\{v_i,\psi_i,\Lambda_i,\Omega_i\}_{i=1}^{k-1}$. In particular, $\tilde {\bf f}_{1}\equiv 0$,
\begin{align*}
\tilde {\bf f}_2 \equiv &  
	(v_1\partial_{R}+\psi_1\partial_\theta) \begin{pmatrix} \Lambda_1\\  \Omega_1  \end{pmatrix}+\frac{\ell }{q}\begin{pmatrix} v_{2-q} \\ 0 \end{pmatrix}
	\\& -\left\{ \left( f_{1,1}+\delta_{1,q}\frac{ \ell R}{q}\right)\partial_{R}+f_{2,1}\partial_\theta 
	+ s_{1}\left(1-\frac{1}{q}+\delta_{1,q}\right)\partial_S- \delta_{1,q}\frac{2-q}{q}\right\} \begin{pmatrix}v_1 \\  \psi_1  \end{pmatrix},
\\
\tilde {\bf f}_3 \equiv &  \sum_{i+j=3} (v_j\partial_{R}+\psi_j\partial_\theta) \begin{pmatrix} \Lambda_i\\  \Omega_i  \end{pmatrix} +\frac{1}{2}\Big( v_1^2 \partial_{R}^2+ 2v_1\psi_1 \partial_{R} \partial_\theta+ \psi_1^2\partial_\theta^2\Big)\begin{pmatrix} \Lambda_1\\  \Omega_1  \end{pmatrix}+\frac{\ell }{q}\begin{pmatrix} v_{3-q} \\ 0 \end{pmatrix}\\
	&
	-	\sum_{j=1}^2 \left\{ \left( f_{1,j}+\delta_{j,q}\frac{ \ell R}{q}\right)\partial_{R}+f_{2,j}\partial_\theta 
	+ s_{j}\left(1-\frac{j}{q}+\delta_{j,q}\right)\partial_S- \delta_{j,q}\frac{3-q}{q}\right\}\begin{pmatrix}v_{3-j} \\  \psi_{3-j}  \end{pmatrix}
	\\
	& 
	-\frac{1}{2}  
\begin{pmatrix}{\hbox{\rm tr}}\left(  {\bf F}^T_{1} {\bf H}_{\bf y}(v_{1}) {\bf F}_{1}\right) \\ {\hbox{\rm tr}}\left({\bf F}^T_{1} {\bf H}_{\bf y}(\psi_{1}){\bf F}_{1}\right) \end{pmatrix} ,
\\
\tilde {\bf f}_k \equiv &
    \sum_{
        \substack{
            l+a_1+\ldots+ia_i+b_1+\ldots+jb_j=k\\
            a_1+\ldots+a_i+b_1+\ldots+b_j\geq 1
                }
            } C_{i,j,a_1,\dots,a_i,b_1,\dots,b_j}   v_1^{a_1}\cdots v_{i}^{a_i}\psi_1^{b_1}  \cdots \psi_{j}^{b_j} \partial^{a_1+\ldots+a_i}_{ R}\partial^{b_1+\ldots+b_j}_\theta  \begin{pmatrix} \Lambda_l \\  \Omega_l  \end{pmatrix}  \\
            & - \sum_{j=1}^{k-1} \left\{   \left( f_{1,j}+\delta_{j,q}\frac{\ell  R}{q}\right)\partial_{ R}+f_{2,j}\partial_\theta
	+  s_{j }\left(1-\frac{j}{ q}+\delta_{j, q}\right)\partial_S -\delta_{j,q}\frac{k-q}{q}\right\}\begin{pmatrix}v_{k-j} \\  \psi_{k-j}  \end{pmatrix}  \\
	& -\frac{1}{2} \sum_{i+j+l=k} 
\begin{pmatrix}{\hbox{\rm tr}}\left( {\bf F}^T_{i} {\bf H}_ {\bf y} (v_{j}) {\bf F}_{l}\right) \\ {\hbox{\rm tr}}\left( {\bf F}^T_{i} {\bf H}_{\bf y}(\psi_{j}) {\bf F}_{l}\right) \end{pmatrix}+\frac{\ell }{q}\begin{pmatrix} v_{k-q} \\ 0 \end{pmatrix},
\end{align*} 
where $C_{i,j,a_1,\dots,a_i,b_1,\dots,b_j}$ are some constant parameters. 

Let us define
\begin{align*}
\begin{pmatrix}
\Lambda_k(R,\theta)\\
\Omega_k(R,\theta)
\end{pmatrix} \equiv \left
\langle  {\bf f}_k( {\bf y},S)-\tilde {\bf f}_k({\bf y},S)\right\rangle_{\varkappa S},
\end{align*}
where
\begin{gather*}
\langle Z_k(R,\theta,S)\rangle_{\varkappa S}:=\frac{1}{2\pi\varkappa}\int\limits_0^{2\pi\varkappa} Z(R,\theta,S)\,dS=\frac{1}{2\pi}\int\limits_0^{2\pi } Z(R,\theta,\varkappa S)\,dS.
\end{gather*}
Then, the right-hand side of \eqref{RSys} is $2\pi \varkappa$-periodic in $S$ with zero average. 
Integrating \eqref{RSys} yields
\begin{gather*}
\begin{pmatrix} v_k(R,\theta,S)\\ \psi_k(R,\theta,S)\end{pmatrix}=-\frac{1}{s_0}\int\limits_0^S \begin{pmatrix} \{f_{1,k}(R,\theta,s)-\tilde f_{1,k}(R,\theta,s)\}_{\varkappa s}   \\ \{f_{2,k}(R,\theta,s)-\tilde{f}_{2,k}(R,\theta,s)\}_{\varkappa s}\end{pmatrix}\, ds+\begin{pmatrix} \tilde v_k(R,\theta)\\ \tilde \psi_k(R,\theta)\end{pmatrix},
\end{gather*}
where 
\begin{gather*}
\{Z(R,\theta,s)\}_{\varkappa s}:=Z(R,\theta,s)-\langle Z(R,\theta,s)\rangle_{\varkappa s},
\end{gather*}
 and the functions $\tilde v_k(R,\theta)$ and $\tilde \psi_k(R,\theta)$ are chosen such that $\langle v_k(R,\theta,S)\rangle_{\varkappa S}\equiv \langle \psi_k(R,\theta,S)\rangle_{\varkappa S}\equiv0$. It can easily be checked that 
\begin{equation*}
\begin{array}{lll}
 \tilde {f}_{1,k}(R,\theta,S)= \mathcal O(R), & \Lambda_k(R,\theta)=\mathcal O(R),  & v_k(R,\theta,S)=\mathcal O(R),\\
\tilde {f}_{2,k}(R,\theta,S)= \mathcal O(1), & \Omega_k(R,\theta)=\mathcal O(1), & \psi_k(R,\theta,S)=\mathcal O(1) 
\end{array}
\end{equation*}
as $R\to 0$ uniformly for all  $(\theta,S)\in\mathbb R^2$. 
 
From \eqref{VnPsiN} it follows that for all $\epsilon\in (0, 1)$  there exist $t_\ast \geq \max\{t_0,1\}$ and $R_0>0$ such that 
\begin{gather}
\label{VNineq}
	|V_N(R,\theta,t)-R|\leq \epsilon R, \quad |\Psi_N(R,\theta,t)-\theta|\leq \epsilon
\end{gather} for all $R\in [0,R_0]$, $\theta\in \mathbb R$ and $t\geq t_\ast$.
Moreover,
\begin{gather*}
\det\frac{\partial(V_N,\Psi_N)}{\partial (R,\theta)}=\begin{vmatrix}
        \partial_{R} V_N(R,\theta,t) & \partial_\theta V_N(R,\theta,t)\\
        \partial_{R} \Psi_N(R,\theta,t) & \partial_\theta \Psi_N(R,\theta,t)
    \end{vmatrix} = 1+\mathcal O(t^{-\frac{1}{q}})
\end{gather*}
as $t\to\infty$ uniformly for all $\theta\in\mathbb R$ and $R\in [0,R_0]$. Hence, the transformation $(R,\theta,t)\to (v,\psi,t)$ is invertible for all $v\in [0,v_0]$, $\psi\in \mathbb R$ and $t\geq t_\ast$ with $v_0=R_0(1-\epsilon)$. We choose $R_0 = E_0^{1/2} t_\ast^{\ell/q}$, then for all $t\geq t_\ast$ the transformation \eqref{Etheta} is valid for all $0\leq E\leq E_0$ and $\varphi\in\mathbb R$.

Denote by $R=\mathfrak R(v,\psi,t)$, $\theta=\mathfrak T(v,\psi,t)$ the inverse transformation to \eqref{exch11}. Then,
\begin{align*}
	 {\bf G}_N({\bf z},t)\equiv &\begin{pmatrix}  \partial_{R} V_N(R,\theta,t) & \partial_\theta V_N(R,\theta,t) \\ \partial_{R} \Psi_N({R},\theta,t) & \partial_\theta \Psi_N({R},\theta,t) \end{pmatrix}  {\bf F}({\bf y},t)\Big|_{R=\mathfrak R(v,\psi,t),\theta=\mathfrak T(v,\psi,t)},\\
	 \tilde \Lambda_N(v,\psi,t) \equiv & -\sum_{k=1}^{N} t^{-\frac{k}{ q}} \Lambda_k(v,\psi)+ \mathcal L({\bf y},{\bf f}({\bf y},t),{\bf F}({\bf y},t)) V_N(R,\theta,t) \Big|_{R=\mathfrak R(v,\psi,t),\theta=\mathfrak T(v,\psi,t)}, 
	\\
	 \tilde \Omega_N(v,\psi,t) \equiv & -\sum_{k=1}^{N} t^{-\frac{k}{ q}}  \Omega_k(v,\psi)+ \mathcal L({\bf y},{\bf f}({\bf y},t),{\bf F}({\bf y},t)) \Psi_N(R,\theta,t) \Big|_{R=\mathfrak R(v,\psi,t),\theta=\mathfrak T(v,\psi,t)}.
\end{align*}
It follows that 
\begin{align*}
 \tilde \Lambda_N=\mathcal O(v)\mathcal O (t^{-\frac{N+1}{ q}} ), \quad 
\tilde \Omega_N =\mathcal O (t^{-\frac{N+1}{ q}} ), \quad 
 \sigma_{1,j} =\mathcal O(v)\mathcal O(t^{-\frac{1}{q}}), \quad  \sigma_{2,j} = \mathcal O(t^{-\frac{1}{q}})
\end{align*}
as $v\to 0$ and $t\to\infty$ uniformly for all $\psi\in\mathbb R$.

Thus, we obtain the proof of Theorem~\ref{Th1} with
\begin{gather*}
\tilde{v}_N (R,\theta,t)\equiv \sum_{k=1}^N t^{-\frac{k}{ q}} v_k (R,\theta,S(t)), \quad \tilde{\psi}_N(R,\theta,t)\equiv \sum_{k=1}^N t^{-\frac{k}{ q}} \psi_k(R,\theta,S(t)).
\end{gather*}

\section{Asymptotic regimes}
\label{AR}
In this section, we discuss possible asymptotic regimes in the perturbed system \eqref{FulSys} associated with the behaviour of the phase of solutions to the truncated system \eqref{DetSys}.

\begin{proof}[Proof of Lemma~\ref{Lem1}]
Since $\Lambda_N(u,\phi,t)=\mathcal O(u)$ as $u\to 0$ uniformly for all $\phi\in\mathbb R$ and $t\geq t_\ast$, we see that the first equation in \eqref{DetSys} has a fixed point $u(t)\equiv 0$. It follows from \eqref{Omas} that the right-hand side of the second equation at $u=0$ has the form 
\begin{gather*}
	\Omega_N(0,\phi,t)\equiv \sum_{k=m}^N t^{-\frac{k}{q}}\omega_{k,0}(\phi)+\tilde \Omega_N(0,\phi,t).
\end{gather*}
Substituting  $u=0$ and $\phi=\phi_0+\eta$ into \eqref{DetSys} yields
\begin{gather}\label{Zeq}
\frac{d\eta}{dt}=t^{-\frac{m}{q}}\mathcal Z_{m}(\eta)+\tilde{\mathcal Z}_{m}(\eta,t),
\end{gather}
where
\begin{align*}
&\mathcal Z_{m}(\eta)\equiv  \omega_{m,0}(\phi_0+\eta)= \eta\left(\vartheta_{m}+\mathcal O(\eta)\right), \quad \eta\to 0,\\
&\tilde{\mathcal Z}_{m}(\eta,t)\equiv \sum_{k=m+1}^N t^{-\frac{k}{q}} \omega_{k,0}(\phi_0+\eta)+\tilde \Omega_N(0,\phi_0+\eta,t)=\mathcal O(t^{- \frac{m+1}{q}}), \quad t\to\infty.
\end{align*}
Hence, there exist $K_0>0$, $t_1\geq t_\ast$ and $\eta_0>0$ such that
\begin{gather}\label{etaineq}
\frac{d|\eta|}{dt}\leq t^{-\frac{m}{q}}\left(-\frac{|\vartheta_{m}|}{2}|\eta|+K_0t^{-\frac{ 1 }{q}}\right)
\end{gather}
for all $t\geq t_1$ and $|\eta|\leq\eta_0$. For all $\varepsilon\in (0,\eta_0)$ define
\begin{gather*}
\delta_\varepsilon=\frac{4 K_0 }{|\vartheta_{m}|}t_\varepsilon^{-\frac{ 1 }{q}}, \quad 
t_\varepsilon=\max\left\{t_1, \left(\frac{8K_0}{\varepsilon |\vartheta_{m}|}\right)^{q}\right\}.
\end{gather*}
Then $d|\eta(t)|/dt<0$ for solutions of \eqref{Zeq} such that $\delta_\varepsilon \leq |\eta(t)|<\varepsilon$ as $t\geq t_\varepsilon$. Hence, any solution of \eqref{Zeq} with initial data $|\eta(t_\varepsilon)|\leq\delta_\varepsilon$ cannot leave the interval $|\eta|\leq \varepsilon$ as $t \geq t_\varepsilon$.

Let $m=q$. Then, by integrating \eqref{etaineq} with respect to $t$, we get
\begin{gather*}
 \eta(t)  =
\begin{cases} 
\mathcal O(t^{-\frac{|\vartheta_{m}|}{2}})+\mathcal O(t^{-\frac{1}{q}}),& \quad \displaystyle \frac{|\vartheta_{m}|}{2}\neq \frac{1}{q},\\
\mathcal O( t^{-\frac{1}{q}}\log t),& \quad \displaystyle  \frac{|\vartheta_{m}|}{2}= \frac{1}{q}.
\end{cases}
\end{gather*}
Hence, there exists a particular solution of equation \eqref{Zeq} such that $\eta(t)=\mathcal O(t^{-c})$ as $t\to\infty$ with $c=\min\{|\vartheta_{m}|,q^{-1}\}/2>0$. 
\end{proof}

\begin{proof}[Proof of Lemma~\ref{Lem2}]
Since $\omega_{m,0}(\phi)\neq 0$, it follows that $0<\omega^-_{m,0}\leq |\omega_{m,0}(\phi)|\leq \omega^+_{m,0}$ for all $\phi\in\mathbb R$ with some non-zero constants $\omega^\pm_{m,0}$. Hence, there exists $t_1\geq t_\ast$ such that 
$|\Omega_N(0,\phi,t)|\geq t^{-\frac{m}{q}} \omega^-_{m,0}/2$ as $t\geq t_1$. It can easily be seen that $u(t)\equiv 0$ satisfies the first equation in \eqref{DetSys} for all $\phi\in\mathbb R$. Therefore, substituting $u=0$ into the second equation in \eqref{DetSys}, yields $d\phi/dt\geq t^{-\frac{m}{q}} \omega^-_{m,0}/2$ as $t\geq t_1$ if $\omega_{m,0}(\phi)>0$ and $d\phi/dt\leq - t^{-\frac{m}{q}} \omega^-_{m,0}/2$ if $\omega_{m,0}(\phi)<0$. Integrating these inequalities, we see that $|\phi(t)|\to \infty$ as $t\to\infty$.
\end{proof}

\section{Stability of phase locking}
\label{SPL}
\begin{proof}[Proof of Theorem~\ref{Th2}]
We take $N=\max\{n,m\}$ in Theorem~\ref{Th1}. From Lemma~\ref{Lem1} it follows that for all $\varepsilon>0$ the corresponding reduced system \eqref{DetSys} has a solution $u_\varepsilon(t)\equiv 0$, $\phi_\varepsilon(t)$ with asymptotics \eqref{zetaas}. Let us show that the solutions of system \eqref{FulSys} remain in some neighbourhood of this trajectory with a high probability. 

Consider auxiliary functions
\begin{gather}\label{U1U2}
\begin{split}
&U_1({\bf x},t)=V_N\left(t^{\frac{\ell}{q}} \sqrt{H(x_1,x_2)}, \Phi(x_1,x_2)-\varkappa^{-1}S(t),t\right), \\
&U_2({\bf x},t)=\Psi_N\left(t^{\frac{\ell}{q}} \sqrt{H(x_1,x_2)}, \Phi(x_1,x_2)-\varkappa^{-1}S(t),t\right),
\end{split}
\end{gather}
where $H(x_1,x_2)$, $\Phi(x_1,x_2)$  and $V_N(R,\theta,t)$, $\Psi_N(R,\theta,t)$ are defined by \eqref{HPhi} and \eqref{VnPsiN}, respectively. It follows from \eqref{Etheta} and \eqref{VNineq} that for all $\epsilon\in (0,1)$ there exists $E_\ast\in(0, E_0]$ such that
\begin{gather}\begin{split} \label{U1U2est}
&(1-\epsilon) \sqrt{H(x_1,x_2)}\leq t^{-\frac{\ell}{q}} U_1({\bf x},t)\leq (1+\epsilon)\sqrt{H(x_1,x_2)},\\  
&\left|U_2({\bf x},t)-\Phi(x_1,x_2)+\varkappa^{-1}S(t)\right|\leq C t^{-\frac{1}{q}}
\end{split}\end{gather}
for all ${\bf x}\in \mathcal D(E_\ast)$ and $t\geq t_\ast$ with some $C={\hbox{\rm const}}>0$.

It is not hard to prove that
\begin{gather*}
\mathcal L({\bf x},{\bf a},{\bf A}) \left(U ({\bf x},t)\right)^2  \equiv  2U ({\bf x},t) \mathcal L({\bf x},{\bf a},{\bf A}) U ({\bf x},t)+{\hbox{\rm tr}}\left({\bf A}^T({\bf x},t) {\bf N}_{\bf x}(U ){\bf A}({\bf x},t)\right)
\end{gather*}
for all smooth functions $U({\bf x},t)$, where
\begin{gather*}
{\bf N}_{\bf x}(U)\equiv \begin{pmatrix}
		(\partial_{x_1} U)^2 & \partial_{x_1}U\partial_{x_2} U \\
		\partial_{x_2}U\partial_{x_1} U & (\partial_{x_2} U)^2
\end{pmatrix}.
\end{gather*}
From the proof of Theorem~\ref{Th1}, we see that
\begin{align*}
\mathcal L({\bf x},{\bf a},{\bf A})U_1({\bf x},t) &\equiv  \mathcal L({\bf y},{\bf f},{\bf F}) V_N(R,\theta,t) \equiv \Lambda_N\left(U_1({\bf x},t), U_2({\bf x},t),t\right),\\
\mathcal L({\bf x},{\bf a},{\bf A})U_2({\bf x},t) &\equiv \mathcal L({\bf y},{\bf f},{\bf F})\Psi_N(R,\theta,t) \equiv \Omega_N\left( U_1({\bf x},t), U_2({\bf x},t),t\right)
\end{align*}
for all ${\bf x}\in\mathcal D(E_0)$ and $t\geq t_\ast$.
Define
\begin{gather*}
Z_1({\bf x},t)\equiv {\hbox{\rm tr}}\left({\bf A}^T({\bf x},t) {\bf N}_{\bf x}(U_1){\bf A}({\bf x},t)\right), \quad 
Z_2({\bf x},t)\equiv {\hbox{\rm tr}}\left({\bf A}^T({\bf x},t) {\bf N}_{\bf x}(U_2){\bf A}({\bf x},t)\right).
\end{gather*}
Then, using \eqref{exch1}, \eqref{Etheta}, \eqref{exch11} and \eqref{classPert}, we get the following estimates:
\begin{gather}\label{Zineq}
\left|Z_1({\bf x},t)\right|\leq \mu^2 C_1 t^{-\frac{2p}{q}}
, \quad
\left|Z_2({\bf x},t)\right|\leq \mu^2 C_2 t^{-\frac{2p}{q}}
\end{gather}
for all $ {\bf x}\in\mathcal D(E_\ast)$ and $t\geq t_\ast$ with some positive constants $C_1$ and $C_2$.

We divide the remainder of the proof into two parts.

1. First, consider the case $n< m$. The Lyapunov function candidate for system \eqref{FulSys} is constructed in the following form: 
\begin{gather}\label{Uast1}
U_\ast({\bf x},t)=  \big(U_1({\bf x},t)\big)^2  + \big(U_2({\bf x},t)-\phi_\varepsilon(t)\big)^2 +   C^2 t^{-\frac{2}{q}}+\mu^2(C_1+ C_2) \zeta_{p}(t)
\end{gather}
with a positive function
\begin{gather}\label{zetap}
\zeta_{p}(t)\equiv \begin{cases}
\displaystyle t_\ast^{-\frac{2p}{q}}\left(\mathcal T+t_s-t\right), & 2p<q,\\
\displaystyle \log\left(\frac{\mathcal T+t_s}{t}\right), & 2p=q,\\
\displaystyle \int\limits_t^{\mathcal T+t_s} \varsigma^{-\frac{2p}{q}}\,d\varsigma, & 2p>q,
\end{cases}
\end{gather}
and some $t_s\geq t_\ast$.  It can easily be verified that
\begin{eqnarray*}
\mathcal L({\bf x},{\bf a},{\bf A}) U_\ast({\bf x},t)& = &
2U_1({\bf x},t)\Lambda_N\big(U_1({\bf x},t), U_2({\bf x},t),t\big)\\
&&+ 2\left(U_2({\bf x},t)-\phi_\varepsilon(t)\right) \left(\Omega_N\left(U_1({\bf x},t), U_2({\bf x},t),t\right) - \phi_\varepsilon'(t)\right) \\
&& + Z_1({\bf x},t)+ Z_2({\bf x},t) -\frac{2C^2}{q}t^{-1-\frac{2}{q}}+\mu^2 (C_1+C_2) \zeta'_p(t)
\end{eqnarray*}
for all $ {\bf x} \in \mathcal D(E_\ast)$ and $t\geq t_\ast$. Note that $\phi_\varepsilon'(t)\equiv \Omega_N(0,\phi_\varepsilon(t),t)$. From \eqref{LO}, \eqref{LnOm}, \eqref{Lnas1} and \eqref{Omas} it follows that
\begin{align*} 
&\mathcal L({\bf x},{\bf a},{\bf A}) U_\ast({\bf x},t)=\\
& 2 t^{-\frac{n}{q}}U_1^2 \left(\lambda_n(\phi_\varepsilon(t))+\mathcal O(\Delta) +\mathcal O(t^{-\frac{1}{q}})\right)\\
& +2 t^{-\frac{m}{q}} \Big(\omega_{m,0}'(\phi_\varepsilon(t)) \big(U_2-\phi_\varepsilon(t)\big)^2+\omega_{m,1}(\phi_\varepsilon(t))U_1 \big(U_2-\phi_\varepsilon(t)\big)+\mathcal O(\Delta^3)\Big)\left(1+\mathcal O(t^{-\frac{1}{q}})\right)\\
&+ Z_1({\bf x},t)+Z_2({\bf x},t)+\mu^2(C_1+C_2) \zeta'_p(t) -\frac{2C^2}{q}t^{-1-\frac{2}{q}}
\end{align*}
as $\Delta({\bf x},t)\to 0$ and $t\to\infty$, where $\Delta({\bf x},t)\equiv \sqrt{(U_1({\bf x},t))^2+(U_2({\bf x},t)-\phi_\varepsilon(t))^2}$. 
The application of Young's inequality yields
\begin{equation}\label{Yineq}
\begin{array}{rcl}
\omega_{m,1}(\phi_\varepsilon(t))U_1 (U_2-\phi_\varepsilon(t))&\leq & \chi_m U_1 |U_2-\phi_\varepsilon(t)|\\
&&\\
& \leq &\displaystyle \frac{|\omega_{m,0}'(\phi_\varepsilon(t))|}{2} (U_2-\phi_\varepsilon(t))^2+\frac{\chi_m^2}{2|\omega_{m,0}'(\phi_\varepsilon(t))|} U_1^2
\end{array}
\end{equation}
with $\chi_m=1+\max_\psi |\omega_{m,1}(\psi)|$. 
By choosing $\varepsilon>0$ in Lemma~\ref{Lem1} small enough, we can ensure that 
\begin{gather*}
\lambda_n(\phi_\varepsilon(t))\leq -\frac{|\lambda_n(\phi_0)|}{2}, \quad  
\omega_{m,0}'(\phi_\varepsilon(t))\leq - \frac{|\vartheta_m|}{2}
\end{gather*} 
for all $t\geq t_\ast$.
From \eqref{Zineq} and \eqref{zetap} it follows that $Z_i({\bf x},t)+\mu^2 C_i \zeta'_p(t)\leq 0$  for all ${\bf x}\in\mathcal D(E_\ast)$ and $t\geq t_\ast$.
Combining these estimates, we see that 
\begin{gather*} 
\mathcal L({\bf x},{\bf a},{\bf A}) U_\ast({\bf x},t)\leq -\left( t^{-\frac{n}{q}}|\lambda_n(\phi_0)|U_1^2 + t^{-\frac{m}{q}}\frac{|\vartheta_m|}{2} (U_2-\phi_\varepsilon(t))^2\right) \left(1+\mathcal O(\Delta)+\mathcal O(t^{-\frac{1}{q}})\right).
\end{gather*}
Therefore, there exist $\Delta_1>0$ and $t_1\geq t_\ast$ such that $\mathcal L({\bf x},{\bf a},{\bf A}) U_\ast({\bf x},t)\leq 0$ for all ${\bf x}\in\mathcal D(E_\ast)$ and $t\geq t_\ast$ such that $\Delta({\bf x},t)\leq \Delta_1$.

By choosing $\epsilon>0$ in \eqref{U1U2est} small enough,  we obtain the following inequalities:
\begin{gather*} 
	\frac{d^2({\bf x},t)}{2}\leq \Delta^2({\bf x},t)+C^2 t^{-\frac{2}{q}}\leq 2 d^2({\bf x},t)+3 C^2 t^{-\frac{2}{q}}
\end{gather*}
for all $ {\bf x} \in \mathcal D(E_\ast)$ and $t\geq t_\ast$. Therefore, 
\begin{gather}\label{Uprop}
	U_\ast({\bf x},t)\geq \frac{d^2({\bf x},t)}{2}, \quad \mathcal L({\bf x},{\bf a},{\bf A}) U_\ast({\bf x},t)\leq 0
\end{gather}
for all $({\bf x},t)\in\mathfrak D(d_1,t_2,\mathcal T)$ with $d_1= \Delta_1/2>0$ and $t_2=\max\{t_1,(2C/\Delta_1)^q\}$,
where
\begin{gather*}
\mathfrak D(d_1,t_2,\mathcal T)\equiv \{({\bf x},t):\, {\bf x}\in\mathcal D(E_\ast),\, d({\bf x},t)\leq d_1, \, t_2\leq t\leq t_2+\mathcal T\}.
\end{gather*}

Fix the parameters $\varepsilon_1\in (0,d_1)$  and $\varepsilon_2>0$. Let ${\bf x}(t)$ be a solution of system \eqref{FulSys} with initial data ${\bf x}_0$ such that $d({\bf x}_0,t_s)<\delta$ and $\tau_{\mathfrak D}$ be the first exit time of $({\bf x}(t), t)$ from the domain $\mathfrak D(\varepsilon_1, t_s, \mathcal T)$ with some $0 < \delta < \varepsilon_1 \leq d_1$ and $t_s\geq t_2$. Define the function $\tau_t = \min\{\tau_{\mathfrak D}, t\}$, then $({\bf x}(\tau_t), \tau_t)$ is the process stopped at the first exit time from the domain $\mathfrak D(\varepsilon_1, t_s, \mathcal T)$. 
From \eqref{Uprop} it follows that $U_\ast({\bf x}(\tau_t), \tau_t)$ is a non-negative supermartingale~\cite[\S 5.2]{RKh12}, and the following estimates hold:
\begin{eqnarray*}
\mathbb P\left(\sup_{0\leq t-t_s\leq \mathcal T} d({\bf x}(t),t)> \varepsilon_1\right)&= &
\mathbb P\left(\sup_{t\geq t_s} d({\bf x}(\tau_t),\tau_t)> \varepsilon_1\right)\\
&\leq & \mathbb P\left(\sup_{t\geq t_s} U_\ast({\bf x}(\tau_t),\tau_t)> \frac{\varepsilon_1^2}{2}\right)\leq \frac{2 U_\ast({\bf x}(t_s),t_s)}{\varepsilon_1^2}.
\end{eqnarray*}
The last estimate follows from Doob’s inequality for supermartingales. Note that $ U_\ast({\bf x}(t_s),t_s)\leq 2\delta^2 + 3 C^2 t_s^{-2/q}+\mu^2 C_3 \zeta_p(t_s)$ with $C_3=C_1+C_2$. Hence, taking $\delta=\varepsilon_1 \sqrt{\varepsilon_2/12}$,
\begin{gather*}
t_s=\begin{cases}
 \max\left\{t_2, \left(\frac{3C^2}{2 \delta^2}\right)^{\frac{q}{2}}\right\}, & 2p\leq q,\\
&\\
 \max\left\{t_2,  \left(\frac{3C^2}{2 \delta^2}\right)^{\frac{q}{2}}, \left(\frac{\mu^2 C q }{2 \delta^2 (2p-q)}\right)^{\frac{q}{2p-q}} \right\}, & 2p>q
\end{cases}
\end{gather*}
and $\mathcal T$ defined by \eqref{mathcalT} with
\begin{gather*} 
C_0=\begin{cases}
2 t_\ast^{\frac{2p}{q}} C_3^{-1}, & 2p< q \\  
2 C_3^{-1}, & 2p=q,
\end{cases}
\end{gather*}
we obtain \eqref{def1} with $\varsigma=0$.

2. Consider the case $n\geq m$. We use 
\begin{gather}\label{Uast2}
U_\ast({\bf x},t)=  \big(U_1({\bf x},t)\big)^2+\mu^2  C_1 \zeta_{p}(t) + c_\ast t^{-\frac{n-m}{q}}\left[\big(U_2({\bf x},t)-\phi_\varepsilon(t)\big)^2 +    C^2 t^{-\frac{2}{q}}+\mu^2  C_2 \zeta_{p}(t)\right]
\end{gather}
with $c_\ast= |\lambda_n(\phi_0)| |\vartheta_m|/(4\chi_m^2)>0$ as the Lyapunov function candidate. It can be proved, as in the previous case, that 
\begin{gather}\label{LUineq2}
\begin{split}
&\mathcal L({\bf x},{\bf a},{\bf A}) U_\ast({\bf x},t)\leq t^{-\frac{n}{q}}U_1^2 \left(-|\lambda_n(\phi_0)|+\mathcal O(\Delta) +\mathcal O(t^{-\frac{1}{q}})\right)\\
& +t^{-\frac{n}{q}}c_\ast \left(-|\vartheta_m| (U_2-\phi_\varepsilon(t))^2+2\chi_m U_1 |U_2-\phi_\varepsilon(t)|+\mathcal O(\Delta^3)\right)\left(1+\mathcal O(t^{-\frac{1}{q}})\right)\\
&+ Z_1({\bf x},t)+\mu^2 C_1 \zeta'_p(t)+c_\ast t^{-\frac{n-m}{q}}\left[Z_2({\bf x},t)+\mu^2 C_2 \zeta'_p(t)\right]
\end{split}
\end{gather}
as $\Delta({\bf x},t)\to 0$ and $t\to\infty$. We see that
\begin{eqnarray}\label{ineqaux2}
2c_\ast\chi_m U_1 |U_2-\phi_\varepsilon(t)|&\leq &   \frac{|\lambda_n(\phi_0)|}{2}  U_1^2+ \frac{c_\ast |\vartheta_m|}{2}  (U_2-\phi_\varepsilon(t))^2
\end{eqnarray}
for all ${\bf x}\in\mathcal D(E_\ast)$ and $t\geq t_\ast$. 
Hence, combining \eqref{Zineq}, \eqref{LUineq2} and \eqref{ineqaux2}, we obtain
\begin{gather*}
\mathcal L({\bf x},{\bf a},{\bf A}) U_\ast({\bf x},t)\leq -t^{-\frac{n}{q}} \frac{1}{4} \left( |\lambda_n(\phi_0)|U_1^2 +  c_\ast|\vartheta_m| (U_2-\phi_\varepsilon(t))^2\right) \left(1+\mathcal O(\Delta)+\mathcal O(t^{-\frac{1}{q}})\right).
\end{gather*} 
Therefore, there exist $\Delta_1>0$ and $t_1\geq t_\ast$ such that $\mathcal L({\bf x},{\bf a},{\bf A}) U_\ast({\bf x},t)\leq 0$ for all ${\bf x}\in\mathcal D(E_\ast)$ and $t\geq t_1$ such that $\Delta({\bf x},t)\leq \Delta_1$.
Thus, 
\begin{gather*} 
	U_\ast({\bf x},t)\geq t^{-\frac{n-m}{q}} \frac{c_-}{2}d^2({\bf x},t), \quad \mathcal L({\bf x},{\bf a},{\bf A}) U_\ast({\bf x},t)\leq 0
\end{gather*}
for all $({\bf x},t)\in\mathfrak D(d_1,t_2,\mathcal T)$,  with $c_-=\min\{1,c_\ast\}>0$, $d_1= \Delta_1/2>0$ and $t_2=\max\{t_1,(2C/\Delta_1)^q\}$. 

Fix the parameters $\varepsilon_1\in (0,d_1)$ and $\varepsilon_2>0$. Then, arguing as above, we obtain the following inequality:
\begin{gather*}
\mathbb P\left(\sup_{0\leq t-t_s\leq \mathcal T} \left\{t^{-\frac{n-m}{2q}} d({\bf x}(t),t)\right\}> \varepsilon_1\right) \leq \frac{2 U_\ast({\bf x}(t_s),t_s)}{\varepsilon_1^2 c_-},
\end{gather*}
where ${\bf x}(t)$ is a solution of system \eqref{FulSys} with initial data ${\bf x}_0$ such that $d({\bf x}_0,t_s)<\delta$ with some $0<\delta<\varepsilon_1$ and $t_s\geq t_2$. It follows easily that $ U_\ast({\bf x}(t_s),t_s)\leq c_+(2\delta^2 + 3 C^2 t_s^{-2/q}+\mu^2 C_3 \zeta_p(t_s))$ with $c_+=\max\{1,c_\ast\}$. Hence, by taking $\delta=\varepsilon_1 \sqrt{\varepsilon_2/(12 c_+)}$, and the parameters $t_s$ and $\mathcal T$ the same as in the previous case, we get \eqref{def1}  with $\varsigma=(n-m)/(2q)$.
\end{proof}

\begin{proof}[Proof of Lemma~\ref{Lem3}]
Substituting  $u=t^{-\kappa}(u_0+\eta_1)$ and $\phi=\phi_0+\eta_2$ into \eqref{DetSys} yields
\begin{gather}\label{eta12}
\frac{d\eta_1}{dt}=\mathcal A(\eta_1,\eta_2,t),
\quad 
\frac{d\eta_2}{dt}=\mathcal B(\eta_1,\eta_2,t),
\end{gather}
where
\begin{align*}
&\mathcal A(\eta_1,\eta_2,t)\equiv t^\kappa\Lambda_N(t^{-\kappa}(u_0+\eta_1),\phi_0+\eta_2,t)+\kappa t^{-1}(u_0+\eta_1),\\
&\mathcal B(\eta_1,\eta_2,t)\equiv \Omega_N(t^{-\kappa}(u_0+\eta_1),\phi_0+\eta_2,t).
\end{align*}
It follows from \eqref{Lnas2}, \eqref{Omas} and \eqref{kappau0} that
\begin{align*}
&\mathcal A(\eta_1,\eta_2,t)=t^{-\frac{n+l}{q}}\left(\mathcal A_1 \eta_1+\mathcal A_2 \eta_2+\mathcal O(\rho^2)+\mathcal O(t^{-\kappa_1})\right),\\
&\mathcal B(\eta_1,\eta_2,t)=t^{-\frac{m}{q}} \left(\vartheta_m \eta_2+\mathcal O(\eta_2^2)+ \mathcal O(t^{-\kappa_1})\right)
\end{align*}
as $t\to\infty$ and $\rho=\sqrt{\eta_1^2+\eta_2^2}\to 0$, where $\mathcal A_1=-(h-1)u_0^{h-1}|\lambda_{n,h}(\phi_0)|$, $\mathcal A_2=\lambda_{n,h}'(\phi_0)u_0^h+\lambda_{n+l}'(\phi_0)u_0$ and $\kappa_1=\min\{\kappa,q^{-1}\}$.

Consider
\begin{gather*}
L(\eta_1,\eta_2,t)=K\eta_2^{2}+\left(\eta_1+\frac{\mathcal A_2}{\mathcal A_1}\eta_2\right)^{2}, \quad K=\frac{|\vartheta_m| |\mathcal A_2|^2}{|\mathcal A_1|^3}
\end{gather*}
as a Lyapunov function candidate for system \eqref{eta12}. 
It can easily be checked that 
\begin{gather*}
K_- C_1^{-1}\rho^{2} \leq L(\eta_1,\eta_2,t)\leq K_+ C_1\rho^2, \quad C_1=2\left(1+\left( {\mathcal A_2} \mathcal A_1 ^{-1}\right)^{2}\right)>0,
\end{gather*}
$K_-=\min\{1,K\}>0$ and $K_+=\max\{1,K\}>0$.
The derivative of $L(\eta_1,\eta_2,t)$ along the
trajectories of system \eqref{eta12} is given by
\begin{gather*}
\frac{dL}{dt}\Big|_{\eqref{eta12}}=t^{-\frac{m}{q}}\left(2K \vartheta_m \eta_2^{2}+2 \vartheta_m \frac{\mathcal A_2}{\mathcal A_1} \left(\eta_1+\frac{\mathcal A_2}{\mathcal A_1}\eta_2\right)\eta_2 +2\mathcal A_1\left(\eta_1+\frac{\mathcal A_2}{\mathcal A_1}\eta_2\right)^2+\mathcal O(\rho^{3})+\mathcal O(t^{-\kappa_1})\right)
\end{gather*}
as $t\to\infty$ and $\rho\to 0$. 
Hence, there exist $K_0>0$, $t_1\geq t_\ast$ and $\rho_0>0$ such that
\begin{gather*}
\frac{dL}{dt}\leq t^{-\frac{m}{q}}\left( -C_2\rho^2+ K_0 t^{-\kappa_1}\right)
\end{gather*}
for all $t\geq t_1$ and $\rho\leq\rho_0$, where $C_2=K_- \min\{|\vartheta_m|,|\mathcal A_1|\}(2C_1)^{-1}>0$. For all $\varepsilon\in (0,\rho_0)$ define
\begin{gather*}
\delta_\varepsilon=\left(\frac{2 K_0}{ C_2t_\varepsilon^{\kappa_1}}\right)^{2}, \quad 
t_\varepsilon=\max\left\{t_1, \left(\frac{4 K_0}{C_2\varepsilon^{2}}\right)^{\frac{1}{\kappa_1}}\right\}.
\end{gather*}
Then, $dL/dt<0$ for solutions of \eqref{eta12} such that $\delta_\varepsilon \leq \sqrt{\eta_1^2(t)+\eta_2^2(t)}<\varepsilon$ as $t\geq t_\varepsilon$. Hence, any solution of \eqref{eta12} with initial data $\sqrt{\eta_1^2(t_\varepsilon)+\eta_2^2(t_\varepsilon)}\leq\delta_\varepsilon$ cannot leave the domain $\{(\eta_1,\eta_2)\in\mathbb R^2: \rho\leq \varepsilon\}$ as $t \geq t_\varepsilon$.
\end{proof}

\begin{proof}[Proof of Theorem~\ref{Th3}]

First, consider the case $\lambda_{n,h}(\phi_0)<0$ and $\lambda_{n+l}(\phi_0)<0$. In this case the proof is similar to that of Theorem~\ref{Th2}. Indeed, let $n+l<m$. Then we use $U_\ast({\bf x},t)$, defined by \eqref{Uast1}, as a Lyapunov function candidate. It follows easily that
\begin{align*} 
&\mathcal L({\bf x},{\bf a},{\bf A}) U_\ast({\bf x},t)=\\
& 2 t^{-\frac{n}{q}}U_1^{p+1} \left(\lambda_{n,h}(\phi_\varepsilon(t))+\mathcal O(\Delta) +\mathcal O(t^{-\frac{1}{q}})\right)+2 t^{-\frac{n+l}{q}}U_1^2 \left(\lambda_{n+l}(\phi_\varepsilon(t))+\mathcal O(\Delta) +\mathcal O(t^{-\frac{1}{q}})\right)\\
& +2 t^{-\frac{m}{q}} \left(\omega_{m,0}'(\phi_\varepsilon(t)) (U_2-\phi_\varepsilon(t))^2+\omega_{m,1}(\phi_\varepsilon(t))U_1 (U_2-\phi_\varepsilon(t))+\mathcal O(\Delta^3)\right)\left(1+\mathcal O(t^{-\frac{1}{q}})\right)\\
&+ Z_1({\bf x},t)+Z_2({\bf x},t)+\mu^2(C_1+C_2) \zeta'_p(t) -\frac{2C^2}{q}t^{-1-\frac{2}{q}}
\end{align*}
as $\Delta({\bf x},t)\to 0$ and $t\to\infty$. Choosing $\varepsilon>0$ in Lemma~\ref{Lem1} small enough ensures that 
\begin{gather*}
\lambda_{n,h}(\phi_\varepsilon(t))\leq -\frac{|\lambda_{n,h}(\phi_0)|}{2}, \quad \lambda_{n+l}(\phi_\varepsilon(t))\leq -\frac{|\lambda_{n+l}(\phi_0)|}{2}, \quad  
\omega_{m,0}'(\phi_\varepsilon(t))\leq - \frac{|\vartheta_m|}{2}
\end{gather*} 
as $t\geq t_\ast$. Taking into account \eqref{Yineq}, we obtain
\begin{align*} 
\mathcal L({\bf x},{\bf a},{\bf A}) U_\ast({\bf x},t)\leq& -\left(t^{-\frac{n}{q}}|\lambda_{n,h}(\phi_0)|U_1^{p+1}+ t^{-\frac{n+l}{q}}|\lambda_{n+l}(\phi_0)|U_1^2 + t^{-\frac{m}{q}}\frac{|\vartheta_m|}{2} (U_2-\phi_\varepsilon(t))^2\right) \\
& \times \left(1+\mathcal O(\Delta)+\mathcal O(t^{-\frac{1}{q}})\right)
\end{align*}
as $\Delta({\bf x},t)\to 0$ and $t\to\infty$. Hence, $U_\ast({\bf x},t)$ satisfies \eqref{Uprop}. By repeating the steps of the proof of Theorem~\ref{Th2}, we obtain \eqref{def1} with $\varsigma=0$.

If $n+l\geq m$, we consider $U_\ast({\bf x},t)$, defined by \eqref{Uast2}, with $c_\ast= |\lambda_{n+l}(\phi_0)| |\vartheta_m|/(4\chi_m^2)>0$ and $t^{-(n+l-m)/q}$ instead of $t^{-(n-m)/q}$ as a Lyapunov function candidate. In this case, 
\begin{align*} 
&\mathcal L({\bf x},{\bf a},{\bf A}) U_\ast({\bf x},t)\leq \\
&   t^{-\frac{n}{q}}U_1^{p+1} \left(-|\lambda_{n,h}(\phi_0)|+\mathcal O(\Delta) +\mathcal O(t^{-\frac{1}{q}})\right)+  t^{-\frac{n+l}{q}}U_1^2 \left(-|\lambda_{n+l}(\phi_0)|+\mathcal O(\Delta) +\mathcal O(t^{-\frac{1}{q}})\right)\\
& + c_\ast t^{-\frac{n+l}{q}} \left(-|\vartheta_m| (U_2-\phi_\varepsilon(t))^2+2 \chi_m U_1 |U_2-\phi_\varepsilon(t)|+\mathcal O(\Delta^3)\right)\left(1+\mathcal O(t^{-\frac{1}{q}})\right)\\
&+ Z_1({\bf x},t)+\mu^2 C_1 \zeta_p'(t)+c_\ast t^{-\frac{n+l-m}{q}} \left[Z_2({\bf x},t)+\mu^2 C_2 \zeta'_p(t)\right]
\end{align*}
as $\Delta({\bf x},t)\to 0$ and $t\to\infty$. Note that inequality \eqref{ineqaux2} holds with $\lambda_{n+l}(\phi_0)$ instead of $\lambda_{n}(\phi_0)$. It follows that 
\begin{align*} 
\mathcal L({\bf x},{\bf a},{\bf A}) U_\ast({\bf x},t)\leq& -\left(t^{-\frac{n}{q}}|\lambda_{n,h}(\phi_0)|U_1^{p+1}+ t^{-\frac{n+l}{q}}\frac{|\lambda_{n+l}(\phi_0)|}{4} U_1^2 + t^{-\frac{m}{q}}\frac{c_\ast|\vartheta_m|}{4} (U_2-\phi_\varepsilon(t))^2\right) \\
& \times \left(1+\mathcal O(\Delta)+\mathcal O(t^{-\frac{1}{q}})\right)
\end{align*}
as $\Delta({\bf x},t)\to 0$ and $t\to\infty$. Hence, there exist $\Delta_1>0$ and $t_1\geq t_\ast$ such that 
\begin{align*}
&	t^{-\frac{n+l-m}{q}}\frac{c_-}{2}d^2({\bf x},t) \leq  U_\ast({\bf x},t)\leq   c_+ \left(2 d^2({\bf x},t)+3 C^2 t^{-\frac{2}{q}}+ \mu^2 C_3\zeta_{2p}(t)\right),\\
&\mathcal L({\bf x},{\bf a},{\bf A}) U_\ast({\bf x},t)\leq 0
\end{align*}
for all ${\bf x}\in\mathcal D(E_\ast)$ such that $\Delta({\bf x},t)\leq \Delta_1$ and $t\geq t_1$ with $c_-=\min\{1,c\}>0$ and $c_+=\max\{1,c\}>0$. Therefore, repeating the proof of Theorem~\ref{Th2}, we get \eqref{def1} with $\varsigma=(n+l-m)/(2q)$.  

Now, let $\lambda_{n,h}(\phi_0)<0$, $\lambda_{n+l}(\phi_0)>0$ and $n+l=m$. 
Consider the functions
	\begin{gather*}
	\tilde U_1({\bf x},t)\equiv t^\kappa\left(U_1({\bf x},t)-u_\varepsilon(t)\right), \quad 
\tilde U_2({\bf x},t)\equiv U_2({\bf x},t)-\phi_\varepsilon(t).
	\end{gather*}
	 It follows from \eqref{Etheta} and \eqref{VNineq} that there exists $E_\ast\leq E_0$ such that
\begin{gather}
\label{tildeU1U2}
\begin{split}  
&\left|\left(\tilde U_1({\bf x},t)t^{-\kappa}+u_\varepsilon(t) \right) t^{-\frac{\ell}{q}}-\sqrt{H(x_1,x_2)}\right|\leq  C t^{-\frac{1}{q}}\sqrt{H(x_1,x_2)},\\  
&\left|\tilde U_2({\bf x},t)+\phi_\varepsilon(t)-\Phi(x_1,x_2)+\varkappa^{-1}S(t)\right|\leq C t^{-\frac{1}{q}}
\end{split}\end{gather}
for all ${\bf x}\in \mathcal D(E_\ast)$ and $t\geq t_\ast$ with some $C={\hbox{\rm const}}>0$.
Define the domain
\begin{gather*}
\widetilde{\mathcal D}(\tilde E_\ast,t_\ast)=\{({\bf x},t): \ \ (x_1,x_2)\in \mathcal B_0, \ \  t\geq t_\ast, \ \ H(x_1,x_2) t^{2 (\kappa+\ell/q)}\leq \tilde E_\ast\}
\end{gather*}
	with some $0<\tilde E_\ast\leq E_\ast t_\ast^{-2 (\kappa+\ell/q)} $. Hence, $\widetilde{\mathcal D}(\tilde E_\ast,t_\ast) \subset \mathcal D(E_\ast)\times\{t\geq t_\ast\}$.  
From \eqref{exch1}, \eqref{Etheta}, \eqref{exch11} and \eqref{classPert} it follows that the functions
\begin{gather*}
\tilde Z_1({\bf x},t)\equiv {\hbox{\rm tr}}\left({\bf A}^T({\bf x},t) {\bf N}_{\bf x}(\tilde U_1){\bf A}({\bf x},t)\right), \quad 
\tilde Z_2({\bf x},t)\equiv {\hbox{\rm tr}}\left({\bf A}^T({\bf x},t) {\bf N}_{\bf x}(\tilde U_2){\bf A}({\bf x},t)\right)
\end{gather*}
satisfies the estimates
\begin{gather*} 
 |\tilde Z_1({\bf x},t) |\leq \mu^2 \tilde C_1 t^{-\frac{2p}{q}} 
, \quad
 |\tilde Z_2({\bf x},t) |\leq \mu^2 \tilde C_2 t^{-\frac{2p}{q}}
\end{gather*}
for all $ ({\bf x},t)\in\widetilde{\mathcal D}(\tilde E_\ast,t_\ast)$ with some positive constants $\tilde C_1$ and $\tilde C_2$.

In this case, the Lyapunov function candidate for system \eqref{FulSys} is constructed in the following form: 
\begin{gather*} 
U_\ast({\bf x},t)=  K \big(\tilde U_1({\bf x},t)\big)^2  + \big(\tilde U_2({\bf x},t)\big)^2 +  K_1 \tilde C t^{-\frac{2}{q}}+\mu^2(K \tilde C_1+ \tilde C_2) \zeta_{p}(t)
\end{gather*}
with 
\begin{gather*}
K=\frac{|\vartheta_m||\mathcal A_1|}{(1+4|\mathcal A_2|)^2}, \quad 
K_1=\min\{1,K\}, \quad 
\tilde C=(1+\tilde E_\ast)C^2,
\end{gather*} 
and a positive function $\zeta_{p}(t)$ defined by \eqref{zetap} with some $t_s\geq t_\ast$.  
It can easily be checked that
\begin{eqnarray*}
\mathcal L({\bf x},{\bf a},{\bf A}) U_\ast({\bf x},t)& = &
2K t^\kappa \tilde U_1({\bf x},t)\Big(\Lambda_N\left(U_1({\bf x},t), U_2({\bf x},t),t\right)-\Lambda_N\left(u_\varepsilon(t), \phi_\varepsilon(t),t\right)\Big)\\
&&+ 2\tilde U_2({\bf x},t)\Big(\Omega_N\left(U_1({\bf x},t), U_2({\bf x},t),t\right) - \Omega_N\left(u_\varepsilon(t), \phi_\varepsilon(t),t\right)\Big) \\
&&+2\kappa K t^{-1} \left(\tilde U_1({\bf x},t)\right)^2  -\frac{2K_1 \tilde C}{q}t^{-1-\frac{2}{q}}\\
&& + K  \tilde Z_1({\bf x},t)+ \tilde Z_2({\bf x},t) +\mu^2 (K \tilde C_1+\tilde C_2) \zeta'_p(t)
\end{eqnarray*}
for all $ {\bf x} \in \mathcal D(E_\ast)$ and $t\geq t_\ast$. 
From \eqref{LO}, \eqref{LnOm}, \eqref{Lnas2}, \eqref{Omas} and the proof of Lemma~\ref{Lem3} it follows that
\begin{align*} 
&\mathcal L({\bf x},{\bf a},{\bf A}) U_\ast({\bf x},t)=\\
& 2 K t^{-\frac{m}{q}}\tilde U_1 \left(\mathcal A_{1,\varepsilon}(t)\tilde U_1+\mathcal A_{2,\varepsilon}(t)\tilde U_2+\mathcal O(\tilde \Delta^2) +\mathcal O(\tilde \Delta t^{-{\kappa_1}})\right)\\
& +2 t^{-\frac{m}{q}}\tilde U_2 \left(\vartheta_{m,\varepsilon}(t)\tilde U_2+\mathcal O(\tilde U_2^2)+\mathcal O(\tilde \Delta t^{-{\kappa_1}})\right)\\
&+ K \tilde Z_1({\bf x},t)+\tilde Z_2({\bf x},t)+\mu^2(K \tilde C_1+\tilde C_2) \zeta'_p(t) -\frac{2K_1 \tilde C}{q}t^{-1-\frac{2}{q}}
\end{align*}
as $\tilde \Delta({\bf x},t)\to 0$ and $t\to\infty$, where $\tilde \Delta({\bf x},t)=\sqrt{(\tilde U_1({\bf x},t))^2+(\tilde U_2({\bf x},t))^2}$, $\kappa_1=\min\{\kappa,q^{-1}\}$ and
\begin{eqnarray*}
\mathcal A_{1,\varepsilon}(t)&\equiv& \lambda_{n+l}(\phi_\varepsilon(t))+\kappa\delta_{n+l,q}+h \lambda_{n,h}(\phi_\varepsilon(t))(u_0+\varepsilon u_1(t))^{h-1}, \\
\mathcal A_{2,\varepsilon}(t)&\equiv& \lambda_{n+l}'(\phi_\varepsilon(t))(u_0+\varepsilon u_1(t))+\lambda_{n,h}'(\phi_\varepsilon(t))(u_0+\varepsilon u_1(t))^h,\\
\vartheta_{m,\varepsilon}(t)&\equiv& \omega_{m,0}'(\phi_\varepsilon(t)).
\end{eqnarray*} 
Note that by choosing $\varepsilon>0$ in Lemma~\ref{Lem3} small enough, we obtain the following inequalities: 
\begin{gather*}
\mathcal A_{1,\varepsilon}(t)\leq -\frac{|\mathcal A_1|}{2}, \quad |\mathcal A_{2,\varepsilon}(t)|\leq \frac{1+4  |\mathcal A_2|}{2}, \quad  
\vartheta_{m,\varepsilon}(t)\leq - \frac{|\vartheta_m|}{2}
\end{gather*} 
for all $t\geq t_\ast$. We see that
\begin{eqnarray}\label{ineqaux3}
K (1+4|\mathcal A_2|) |\tilde U_1 \tilde U_2| & \leq & 
\frac{K|\mathcal A_1|}{2}  \tilde U_1^2+ \frac{ |\vartheta_m|}{2}  \tilde U_2^2
\end{eqnarray}
for all ${\bf x}\in\tilde {\mathcal D}(E_\ast)$ and $t\geq t_\ast$. 
Combining \eqref{Zineq}, \eqref{LUineq2} and \eqref{ineqaux3}, we obtain
\begin{gather*}
\mathcal L({\bf x},{\bf a},{\bf A}) U_\ast({\bf x},t)\leq -t^{-\frac{m}{q}} \frac{1}{4} \left(K |\mathcal A_1|\tilde U_1^2 +   |\vartheta_m| \tilde U_2^2 \right) \left(1+\mathcal O(\tilde \Delta)+\mathcal O(t^{-\kappa_1})\right)
\end{gather*}  
as $\tilde \Delta({\bf x},t)\to 0$ and $t\to\infty$. Hence, there exist $\tilde \Delta_1>0$ and $t_1\geq t_\ast$ such that 
$
	\mathcal L({\bf x},{\bf a},{\bf A}) U_\ast({\bf x},t)\leq 0
$
for all ${\bf x}\in\mathcal D(E_\ast)$ and $t\geq t_1$ such that $\tilde\Delta({\bf x},t)\leq \tilde\Delta_1$.
It follows from \eqref{tildeU1U2} that
\begin{gather*}
\frac{\tilde d^2({\bf x},t)}{2}\leq \tilde \Delta^2({\bf x},t)+\tilde C t^{-\frac{2}{q}}\leq 2 \tilde d^2({\bf x},t)+5 \tilde C t^{-\frac{2}{q}}
\end{gather*}
for all $({\bf x},t)\in\widetilde {\mathcal D}(\tilde E_\ast,t_\ast)$. Therefore, 
\begin{gather*} 
	U_\ast({\bf x},t)\geq  \frac{K_1}{2} \tilde d^2({\bf x},t), \quad \mathcal L({\bf x},{\bf a},{\bf A}) U_\ast({\bf x},t)\leq 0
\end{gather*}
for all $({\bf x},t)\in\tilde {\mathfrak D}(d_1,t_2,\mathcal T)$ with $d_1= \tilde \Delta_1/2>0$ and $t_2=\max\{t_1,(2\sqrt{2\tilde C}/\tilde \Delta_1)^q\}$,
where
\begin{gather*}
\tilde{\mathfrak D}(d_1,t_2,\mathcal T)=\{({\bf x},t)\in \tilde {\mathcal D}(E_\ast,t_\ast): \ \ \tilde d({\bf x},t)\leq d_1, \ \ t_2\leq t\leq t_2+\mathcal T\}.
\end{gather*}

Let us fix the parameters $\varepsilon_1\in (0,d_1)$  and $\varepsilon_2>0$. 
Then, arguing as in the proof of Theorem~\ref{Th2}, we obtain the following inequality:
\begin{eqnarray*}
\mathbb P\left(\sup_{0\leq t-t_s\leq \mathcal T} \tilde d({\bf x}(t),t)> \varepsilon_1\right)\leq \frac{2 U_\ast({\bf x}(t_s),t_s)}{K_1 \varepsilon_1^2},
\end{eqnarray*}
where ${\bf x}(t)$ is a solution of system \eqref{FulSys} with initial data ${\bf x}_0$ such that $\tilde d({\bf x}_0,t_s)<\delta$ with some $0 < \delta < \varepsilon_1 \leq d_1$ and $t_s\geq t_2$. Note that $ U_\ast({\bf x}(t_s),t_s)\leq K_2 ( 2 \delta^2+ 5  \tilde C t_s^{-{2}/{q}}+\mu^2\tilde C_3 \zeta_{p}(t_s))$ with $\tilde C_3=\tilde C_1+\tilde C_2$. Hence, taking $\delta=\varepsilon_1 \sqrt{K_1\varepsilon_2/(12 K_2)}$,
\begin{gather*}
t_s=\begin{cases}
 \max\left\{t_2, \left(\frac{5 \tilde C^2}{2 \delta^2}\right)^{\frac{q}{2}}\right\}, & 2p\leq q,\\
&\\
 \max\left\{t_2,  \left(\frac{5 \tilde C^2}{2 \delta^2}\right)^{\frac{q}{2}}, \left(\frac{\mu^2 \tilde C_3 q }{2 \delta^2 (2p-q)}\right)^{\frac{q}{2p-q}} \right\}, & 2p>q
\end{cases}
\end{gather*}
and $\mathcal T$ defined by \eqref{mathcalT} with
\begin{gather*} 
C_0=\begin{cases}
2 t_\ast^{\frac{2p}{q}} \tilde C_3^{-1}, & 2p< q \\ 
2 \tilde C_3^{-1}, & 2p=q,
\end{cases}
\end{gather*}
we obtain \eqref{deftilde}.
\end{proof}

\section{Stability of phase drifting}
\label{SPD}
\begin{proof}[Proof of Theorem~\ref{Th4}]
We take $N=\max\{n,m\}$ in Theorem~\ref{Th1} and consider the functions $U_1({\bf x},t)$ and $U_2({\bf x},t)$ defined by \eqref{U1U2}. It follows from \eqref{H0as}, \eqref{Etheta} and \eqref{U1U2est} that 
for all $\epsilon\in (0,1)$ there exists $r_\ast\leq r_0$ such that
\begin{gather*}
(1-\epsilon)^2\frac{|{\bf x}|}{\sqrt 2}\leq t^{-\frac{\ell}{q}}U_1({\bf x},t)\leq (1+\epsilon)^2\frac{|{\bf x}|}{\sqrt 2}
\end{gather*} 
for all $|{\bf x}|\leq r_\ast$ and $t\geq t_\ast$.
Taking $U_1({\bf x},t)$ as the Lyapunov function for system \eqref{FulSys}, it can be seen that 
\begin{gather*}
\mathcal L({\bf x},{\bf a},{\bf A})U_1({\bf x},t)  \equiv  \mathcal L({\bf y},{\bf f},{\bf F}) V_N(R,\theta,t) \equiv \Lambda_N\left(U_1({\bf x},t), U_2({\bf x},t),t\right).
\end{gather*} 
From \eqref{LO}, \eqref{LnOm} and \eqref{Lnas1} it follows that
\begin{gather*} 
 \mathcal L({\bf x},{\bf a},{\bf A}) U_1({\bf x},t)=t^{-\frac{n}{q}}U_1({\bf x},t) \left(\lambda_n(U_2({\bf x},t) )+\mathcal O(U_1({\bf x},t)) +\mathcal O(t^{-\frac{1}{q}})\right)
\end{gather*}
as $U_1({\bf x},t)\to 0$ and $t\to\infty$. Note that there exists $\lambda_n^\ast>0$ such that $\lambda_n(\psi)\leq -\lambda_n^\ast$ for all $\psi\in\mathbb R$. Hence, there exist $t_s\geq t_\ast$ and $0<U_s\leq  r_\ast t_s^{\ell/q}(1-\epsilon)^2/\sqrt 2$ such that
\begin{gather}\label{U1ets}
 \mathcal L({\bf x},{\bf a},{\bf A}) U_1({\bf x},t)\leq -t^{-\frac{n}{q}} U_1({\bf x},t) \frac{\lambda_n^\ast}{2}\leq 0
\end{gather} 
for all $|{\bf x}|\leq r_\ast$ and $t \geq t_s$ such that $ 0 \leq U_1({\bf x},t)\leq U_s$.

Fix the parameters $0<\varepsilon_1<r_\ast t_s^{\ell/q}$ 	and $\varepsilon_2>0$. Let ${\bf x}(t)$ be a solution of system \eqref{FulSys} with initial data $|{\bf x}(t_s)|\leq \delta$, $0<\delta<\varepsilon_1$ and $\tau_{\varepsilon_1}$ be the first exit time of $\tilde {\bf x}(t)\equiv {\bf x}(t)t^{\ell/q}$ from the domain $\mathcal B_{\varepsilon_1}=\{|{\bf x}|\leq \varepsilon_1\}$ as $t\geq t_s$. Define the function $\tau_t = \min\{\tau_{\varepsilon_1}, t\}$, then $\tilde {\bf x}(\tau_t)$ is the process stopped at the first exit time from  $\mathcal B_{\varepsilon_1}$.  From \eqref{U1ets} it follows that $U_1({\bf x}(\tau_t), \tau_t)$ is a non-negative supermartingale, and the following estimates hold:
\begin{eqnarray*}
\mathbb P\left(\sup_{t\geq t_s} |\tilde {\bf x}(t)|> \varepsilon_1\right)&= &
\mathbb P\left(\sup_{t\geq t_s}  |\tilde {\bf x}(\tau_t)|>\varepsilon_1\right)\\
&\leq & \mathbb P\left(\sup_{t\geq t_s} U_1({\bf x}(\tau_t),\tau_t)> \frac{(1-\epsilon)^2\varepsilon_1}{\sqrt 2}\right)\leq \frac{\sqrt 2 U_1({\bf x}(t_s),t_s)}{(1-\epsilon)^2 \varepsilon_1}.
\end{eqnarray*}
Note that $U_1({\bf x}(t_s),t_s)\leq t_s^{\ell/q} (1+\epsilon)^2\delta/\sqrt 2$. Hence, by taking $\delta=\varepsilon_1 \varepsilon_2 [(1-\epsilon)/(1+\epsilon)]^2 t_s^{-\ell/q}$, we obtain \eqref{defTh4}.
\end{proof}

\begin{proof}[Proof of Theorem~\ref{Th5}] 
The proof is similar to that of Theorem~\ref{Th4}. As above, we use $U_1({\bf x},t)$ as the Lyapunov function candidate for system \eqref{FulSys}. In this case, from \eqref{LO}, \eqref{LnOm} and \eqref{Lnas2} it follows that
\begin{eqnarray*} 
 \mathcal L({\bf x},{\bf a},{\bf A}) U_1({\bf x},t)&=&t^{-\frac{n}{q}}U_1^p({\bf x},t) \left(\lambda_{n,h}(U_2({\bf x},t) )+\mathcal O(U_1({\bf x},t)) +\mathcal O(t^{-\frac{1}{q}})\right) \\
& & +t^{-\frac{n+l}{q}}U_1({\bf x},t) \left(\lambda_{n+l}(U_2({\bf x},t) )+\mathcal O(U_1({\bf x},t)) +\mathcal O(t^{-\frac{1}{q}})\right)
\end{eqnarray*}
as $U_1({\bf x},t)\to 0$ and $t\to\infty$. Note that $\lambda_{n,h}(\psi)\leq -\lambda_{n,h}^\ast$ and $\lambda_{n+l}(\psi)\leq -\lambda_{n+l}^\ast$ for all $\psi\in\mathbb R$ with some constants $\lambda_{n,h}^\ast>0$ and  $\lambda_{n+l}^\ast>0$. Hence, there exist $t_s\geq t_\ast$ and $U_s>0$ such that $ \mathcal L({\bf x},{\bf a},{\bf A}) U_1({\bf x},t)\leq 0$ for all $|{\bf x}|\leq r_\ast$ and $t \geq t_s$ such that $ 0 \leq U_1({\bf x},t)\leq U_s$. Finally, by repeating the next steps of the proof of Theorem~\ref{Th4}, we obtain \eqref{defTh4}.
\end{proof}

\section{Examples}\label{SEx}
Consider the examples of asymptotically autonomous stochastic systems with damped oscillatory coefficients and discuss the application of the proposed theory.

\subsection{Example 1}
First, consider a linear stochastic system
\begin{gather}\label{Ex1}
\begin{split}
&dx_1=x_2\,dt, \\
&dx_2=\left (-x_1+t^{-1}\left[a\big(S(t)\big)x_1+b\big(S(t)\big)x_2\right]\right)\,dt+t^{-\frac{p}{2}}c\big(S(t)\big)x_1\,dw_2(t), \quad t\geq 1,
\end{split}
\end{gather}
where 
\begin{gather*}
a(S)\equiv a_0+a_1\cos S,\quad 
b(S)\equiv b_0+b_1\cos S,\quad 
c(S)\equiv c_0+c_1\cos S,\quad 
S(t)\equiv s_0 t+s_2\log t
\end{gather*}
 with $p\in\mathbb Z_{+}$ and constant parameters $a_k$, $b_k$, $c_k$ and $s_k$. Note that system \eqref{Ex1} is of the form \eqref{FulSys} with $q=2$, 
\begin{gather*}
{\bf a}({\bf x},t)\equiv {\bf a}_0({\bf x})+t^{-1}{\bf a}_2({\bf x},S(t)), \quad 
{\bf A}({\bf x},t)\equiv t^{-\frac{p}{2}}{\bf A}_p({\bf x},S(t)),\\
{\bf a}_0({\bf x})\equiv \begin{pmatrix}x_2\\-x_1\end{pmatrix}, \quad 
{\bf a}_2({\bf x},S)\equiv \begin{pmatrix}0\\a(S)x_1+b(S)x_2\end{pmatrix}, \quad 
{\bf A}_p({\bf x},S)\equiv \begin{pmatrix}0&0\\0& c(S)x_1\end{pmatrix}.
\end{gather*}
The corresponding limiting system \eqref{LimSys} with $H(x_1,x_2)\equiv |{\bf x}|^2/2$ has a stable equilibrium at the origin $(0,0)$ and $2\pi$-periodic solutions with $\nu(E)\equiv 1$. Moreover, system \eqref{Ex0}, considered as an example in Section~\ref{PS}, has the form \eqref{Ex1} with $a_0=a_1=b_1=c_0=s_2=0$.  

1. Let $p=1$. Then the transformations of variables described in Theorem~\ref{Th1} and constructed in Section~\ref{ChOfVar} with $\ell=0$, $N=2$, $\xi_1(\phi,E)\equiv \sqrt{2E}\cos \phi$, $\xi_2(\phi,E)\equiv -\sqrt{2E}\sin \phi$,  $\tilde v_2(R,\theta,t)\equiv t^{-1}v_2(R,\theta,S(t))$, $\tilde \psi_2(R,\theta,t)\equiv t^{-1}\psi_2(R,\theta,S(t))$, and
\begin{align*}
&v_2 \equiv    R\left\{\int\limits_0^S \left\{ \frac{a(s)}{2} \sin\left(2\theta+\frac{2s}{\varkappa}\right)-b(s)   \sin^2\left(\theta+\frac{s}{\varkappa}\right)-\frac{c^2(s)}{2}  \cos^4\left(\theta+\frac{s}{\varkappa}\right)\right\}_{\varkappa s} ds\right\}_{\varkappa S}, 
\\ 
&\psi_2 \equiv  \left\{\int\limits_0^S \left\{a(s) \cos^2\left(\theta+\frac{s}{\varkappa}\right)-\frac{b(s)}{2} \sin\left(2\theta+\frac{2s}{\varkappa}\right) -\frac{c^2(s)}{4}\partial_\theta \left(\cos^4\left(\theta+\frac{s}{\varkappa}\right)\right)\right\}_{\varkappa s} ds\right\}_{\varkappa S}
\end{align*}
reduce system \eqref{Ex1} to \eqref{EQs} with $\Lambda_1(v,\psi)\equiv \Omega_1(v,\psi)\equiv 0$, $\Lambda_2(v,\psi)\equiv  \lambda_2(\psi)v$, $\Omega_2(v,\psi)\equiv \omega_{2,0}(\psi)$, where
the form of the functions $\lambda_2(\psi)$ and $\omega_{2,0}(\psi)$ depends on the value of the parameter $s_0$. Hence, the transformed system satisfies \eqref{LnOm}, \eqref{Lnas1} and \eqref{Omas} with $n=m=2$. Moreover, assumption \eqref{classPert} holds with $\mu= |c_0|+|c_1| $ and $2p/q=1$.

Consider first the case $s_0=1$. Then it follows from \eqref{rescond} that $\varkappa=1$ and 
\begin{gather*}
\lambda_2(\psi)\equiv \frac{1}{32}\left(16b_0+6c_0^2+3c_1^2+2c_1^2\cos 2\psi\right), \quad
\omega_{2,0}(\psi)\equiv -\frac{1}{16} \left(8a_0+16s_2+c_1^2\sin 2\psi \right).
\end{gather*}
It is readily seen that if $|8a_0+16s_2|< c_1^2$, then assumption \eqref{zerom} holds with
\begin{gather*}
\phi_0=\frac{1}{2}\arcsin\left(-\frac{8a_0+16s_2}{c_1^2}\right)+\pi j, \quad j\in\mathbb Z,\quad \vartheta_2=-\frac{c_1^2}{8} \cos 2\phi_0<0.
\end{gather*}
It follows that Theorem~\ref{Th2} is applicable with $n=m=q=2p=2$ and
\begin{gather*}
d({\bf x},t)\equiv \sqrt{\frac{|{\bf x}|^2}{2}+\left|\Phi(x_1,x_2) -S(t)-\phi_\varepsilon(t)\right|} 
\end{gather*}
with $\phi_\varepsilon(t)=\phi_0+o(1)$ as $t\to\infty$, where $\tan \Phi(x_1,x_2)\equiv -x_2/x_1$. Hence, if $b_0<b_\ast$ with
\begin{gather*}
b_\ast=-\frac{6c_0^2+3c_1^2}{16}-\frac{1}{8}\sqrt{c_1^4-(8a_0+16s_2)^2},
\end{gather*}
then the phase locking occurs in system \eqref{Ex1} and the equilibrium $(0,0)$ is stable on exponentially long time interval as $\mu\to 0$ (see~Fig.~\ref{FigEx11}). By applying this result to the example from Section~\ref{PS}, we see that the phase locking occurs in system \eqref{Ex0} and the equilibrium is stable if $b<b_\ast$, where $b_\ast=-5/16$.
\begin{figure}
\centering
   \includegraphics[width=0.4\linewidth]{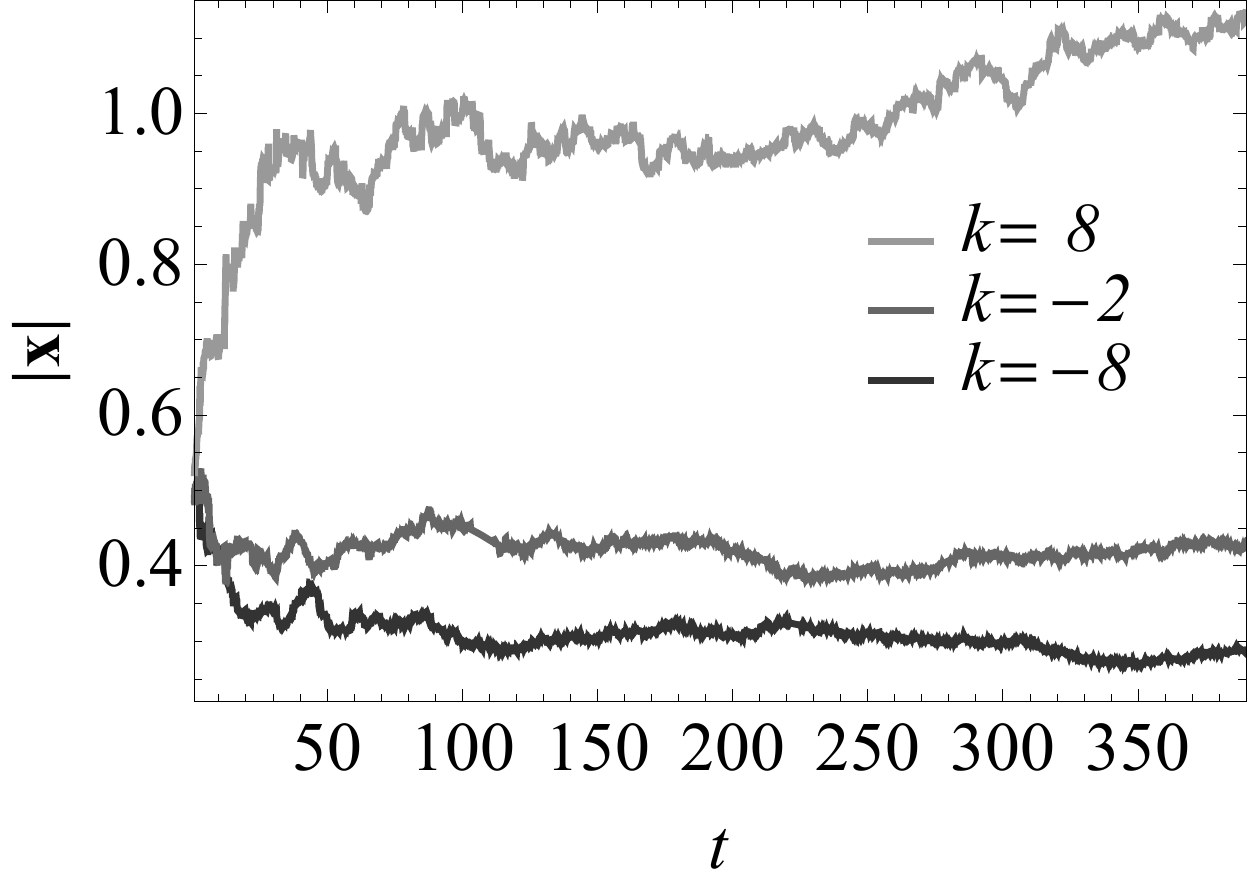}
\hspace{1ex}
   \includegraphics[width=0.4\linewidth]{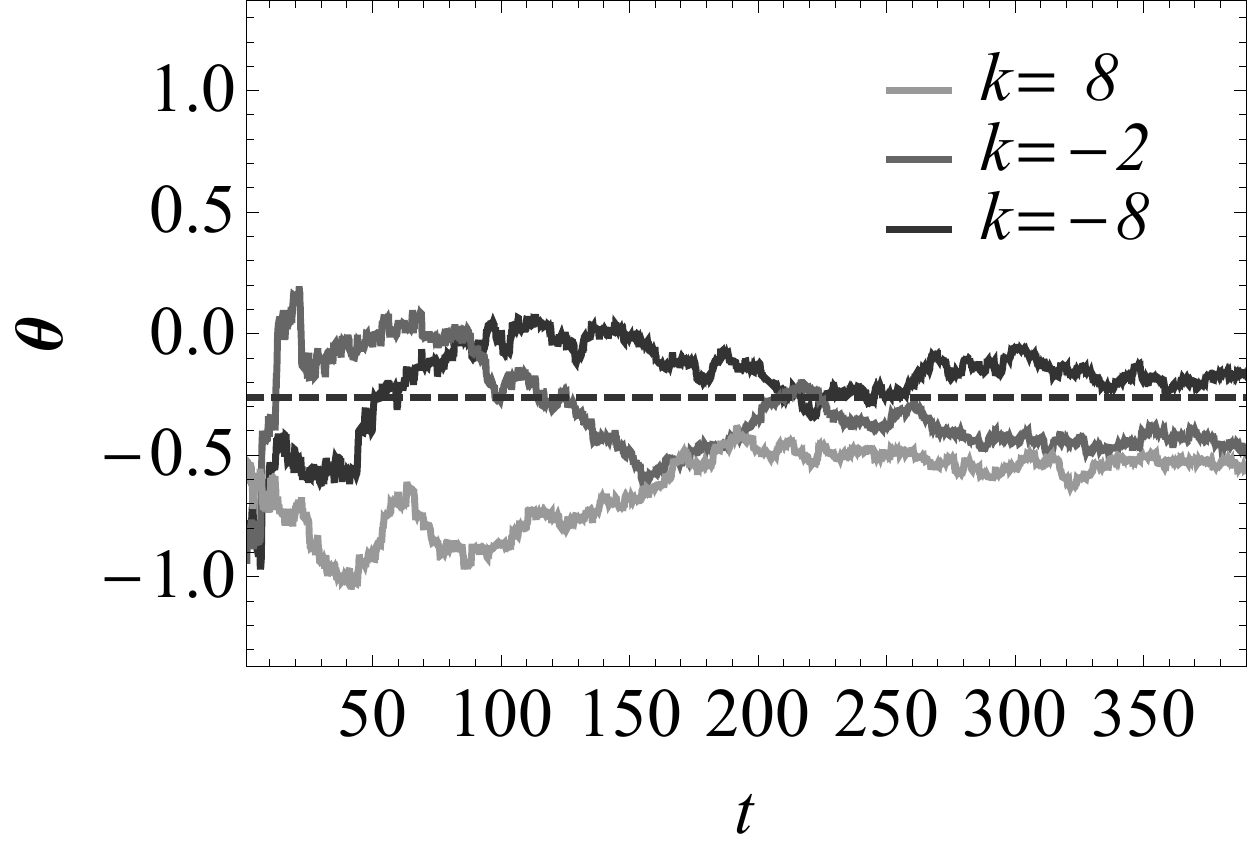}
\caption{\small The evolution of $|{\bf x}(t)|$ and $\theta(t)=\Phi(x_1(t),x_2(t))-S(t)$ for sample paths of solutions to system \eqref{Ex1} with $a_0=\mu^2/16$, $b_0=k\mu^2/16$, $c_1=\mu$, $s_0=1$, $a_1=b_1=c_0=s_2=0$, $\mu=0.5$. The black dashed curve corresponds to $\theta(t)\equiv \phi_0$, where $\phi_0=-\pi/12$. In this case, $b_0<b_\ast$ if and only if $k<-(3+\sqrt 3)\approx -4.73$.} \label{FigEx11}
\end{figure}

If $|8a_0+16s_2|> c_1^2$, then assumption \eqref{nzerom} holds and Theorem~\ref{Th4} is applicable with $m=q=2$. Hence, the phase drifting regime is realized, and the equilibrium $(0,0)$ is stable if $b_0<b_\ast$ with $b_\ast=-(6c_0^2+5c_1^2)/16$ (see~Fig.~\ref{FigEx12}).
\begin{figure}
\centering
  \includegraphics[width=0.4\linewidth]{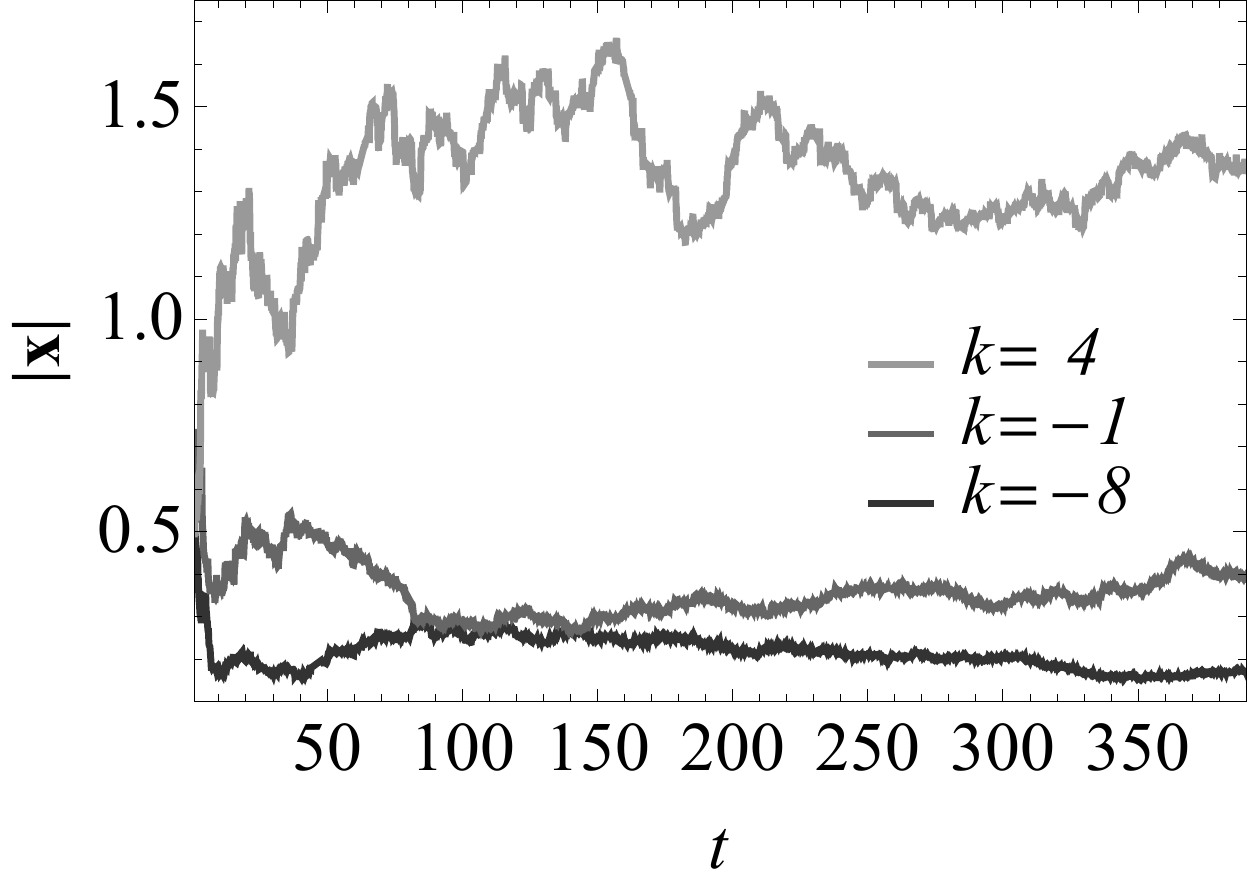}
\hspace{1ex}
  \includegraphics[width=0.4\linewidth]{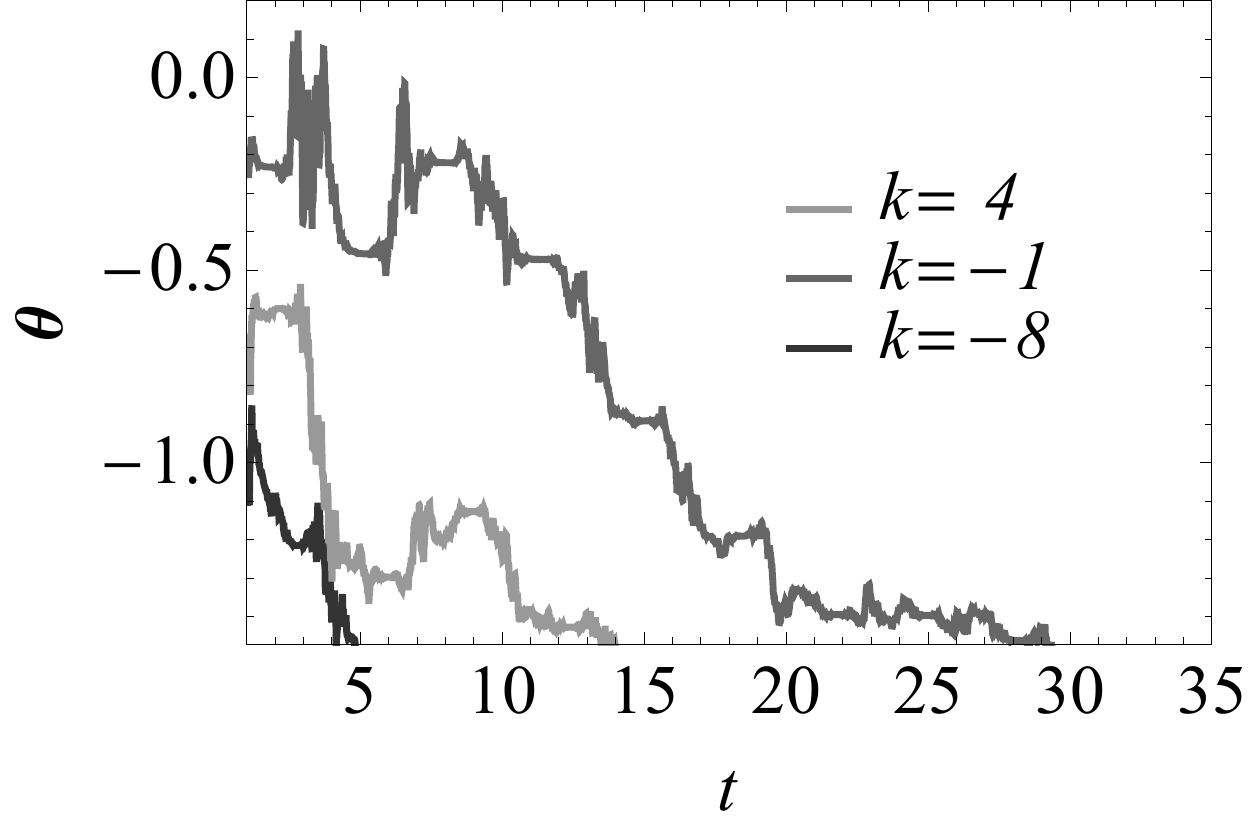}
\caption{\small The evolution of $|{\bf x}(t)|$ and $\theta(t)=\Phi(x_1(t),x_2(t))-S(t)$ for sample paths of solutions to system \eqref{Ex1} with $a_0=1/2$, $b_0=k/16$, $c_1=1$, $s_0=1$, $a_1=b_1=c_0=s_2=0$. In this case, $b_0<b_\ast$ if and only if $k<-5$.} \label{FigEx12}
\end{figure}

Consider now the case $s_0=2$. We have  $\varkappa=2$ and 
\begin{gather*}\lambda_2(\psi)\equiv \frac{1}{64}\left(32b_0+12c_0^2+6c_1^2+c_1^2\cos 4\psi-16 (a_1 \sin2\psi+(b_1-c_0c_1)\cos2\psi)\right), \\
\omega_{2,0}(\psi)\equiv -\frac{1}{32} \left(16(a_0+s_2)+c_1^2\sin 4\psi+8 (a_1 \cos2\psi-(b_1-c_0c_1)\sin2\psi) \right).
\end{gather*}
If $a_1=b_1=c_0=0$ and $16|a_0+s_2|<c_1^2$, then condition \eqref{zerom} holds with
\begin{gather*}
\phi_0=\frac{1}{4}\arcsin\left(-\frac{16(a_0+s_2)}{c_1^2}\right)+\frac{\pi j}{2}, \quad j\in\mathbb Z,\quad \vartheta_2=-\frac{c_1^2}{8} \cos 4\phi_0<0.
\end{gather*}
Moreover, if $b<b_\ast$ with
\begin{gather*}
b_\ast=-\frac{3c_1^2}{16}-\frac{1}{32}\sqrt{c_1^4-256(a_0+s_2)^2},
\end{gather*}
then from Theorem~\ref{Th2} with $n=m=q=2p=2$, and
\begin{gather*}
d({\bf x},t)\equiv \sqrt{\frac{|{\bf x}|^2}{2}+\left|\Phi(x_1,x_2)-\frac{S(t)}{2}-\phi_\varepsilon(t)\right|}
\end{gather*}
it follows that the phase locking regime is realized, and the equilibrium $(0,0)$ is stable.
If $a_1=b_1=c_0=0$ and $16|a_0+s_2|>c_1^2$, then condition \eqref{nzerom} holds. If, in addition, $b_0<- 7c_1^2/32$, then it follows from Theorem~\ref{Th4} with $m=q=2$ that the phase drifting mode takes place and the equilibrium $(0,0)$ is stable.

Consider also the case $s_0=3$. We take $\varkappa=3$ and obtain
\begin{gather*}
\lambda_2(\psi)\equiv \frac{16b_0+6c_0^2+3c_1^2}{32}, \quad \omega_{2,0}(\psi)\equiv -\frac{3a_0+2s_2}{6}.
\end{gather*}
Clearly, if $3a_0+2s_2\neq 0$, then assumption \eqref{nzerom} holds. In this case, the phase drifting regime occurs. Moreover, if $b_0<-(6c_0^2+3c_1^2)/16$, then from Theorem~\ref{Th4} it follows that the equilibrium $(0,0)$ is stable in system \eqref{Ex1}.

2. Now let $p=2$. Then the change of variables described in Theorem~\ref{Th1} with $\ell=0$, $N=1$, $\xi_1(\phi,E)\equiv \sqrt{2E}\cos \phi$, $\xi_2(\phi,E)\equiv -\sqrt{2E}\sin \phi$,  $\tilde v_1(R,\theta,t)\equiv t^{-1/2}v_1(R,\theta,S(t))$, $\tilde \psi_1(R,\theta,t)\equiv t^{-1/2}\psi_1(R,\theta,S(t))$, and
\begin{align*}
&v_1 \equiv    R\left\{\int\limits_0^S \left\{ \frac{a(s)}{2} \sin\left(2\theta+\frac{2s}{\varkappa}\right)-b(s)   \sin^2\left(\theta+\frac{s}{\varkappa}\right) \right\}_{\varkappa s} ds\right\}_{\varkappa S}, 
\\ 
&\psi_1 \equiv  \left\{\int\limits_0^S \left\{a(s) \cos^2\left(\theta+\frac{s}{\varkappa}\right)-\frac{b(s)}{2} \sin\left(2\theta+\frac{2s}{\varkappa}\right) \right\}_{\varkappa s} ds\right\}_{\varkappa S}
\end{align*}
transforms system \eqref{Ex1} to \eqref{EQs} with $\Lambda_1(v,\psi)\equiv \lambda_1(\psi)v$, $\Omega_1(v,\psi)\equiv \omega_{1,0}(\psi)$. We see that the system satisfies \eqref{LnOm}, \eqref{Lnas1} and \eqref{Omas} with $n=m=1$ and assumption \eqref{classPert} holds with $\mu= |c_0|+|c_1| $ and $2p/q=2$.

Consider the case $s_0=1$. Then $\varkappa=1$ and 
\begin{gather*}
\lambda_1(\psi)\equiv \frac{b_0}{2}, \quad
\omega_{1,0}(\psi)\equiv -\frac{a_0+2s_2}{2}.
\end{gather*}
It is readily seen that if $a_0+2s_2\neq 0$, then assumption \eqref{nzerom} holds and Theorem~\ref{Th4} is applicable with $m<q$. Hence, the phase drifting regime takes place, and the equilibrium $(0,0)$ is stable if $b_0<0$ (see Fig.~\ref{Fig12}, b).
 
In the case $s_0=2$, we have $\varkappa=2$ and
\begin{gather*}
\lambda_1(\psi)\equiv \frac{1}{4}\left(2b_0-d_1 \sin (2\psi+\delta_1)\right), \quad
\omega_{1,0}(\psi)\equiv -\frac{1}{4}\left(a_0+s_2+8d_1\cos(2\psi+\delta_1)\right),
\end{gather*}
where $d_1=\sqrt{a_1^2+b_1^2}$ and $\delta_1=\arccos(a_1/d_1)$.
If $|a_0+s_2|< 8 d_1$, then assumption \eqref{zerom} holds with
\begin{gather*}
\phi_0=-\frac{\delta_1}{2}-\frac{1}{2}\arccos\left(-\frac{a_0+s_2}{8d_1}\right)+\pi j, \quad j\in\mathbb Z,\quad \vartheta_1=4d_1 \sin(2\phi_0+\delta_1)<0,
\end{gather*}
and Theorem~\ref{Th2} is applicable with $n=m<q<2p$. Hence, the phase locking regime is realized  (see~Fig.~\ref{FigEx13}), and the equilibrium $(0,0)$ is stable  if 
\begin{gather*}
b_0<-\frac{1}{16}\sqrt{64d_1^2-(a_0+s_2)^2}.
\end{gather*}
If $|a_0+s_2|> 8 d_1$, then assumption \eqref{nzerom} holds. Hence, it follows from Theorem~\ref{Th4} that the phase drifting regime occurs, and the equilibrium $(0,0)$ is stable if $b_0<-d_1/2$.

\begin{figure}
\centering
  \includegraphics[width=0.4\linewidth]{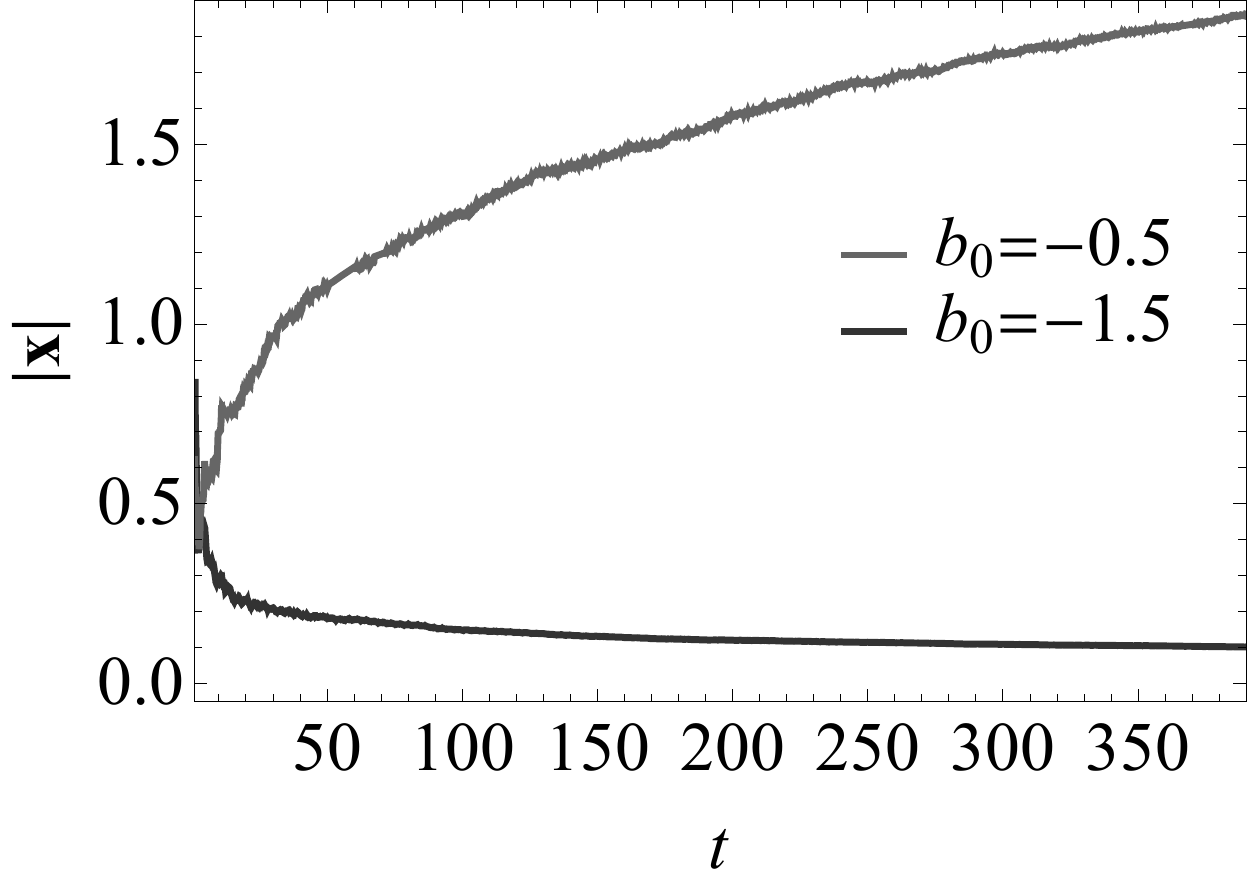}
\hspace{1ex}
  \includegraphics[width=0.4\linewidth]{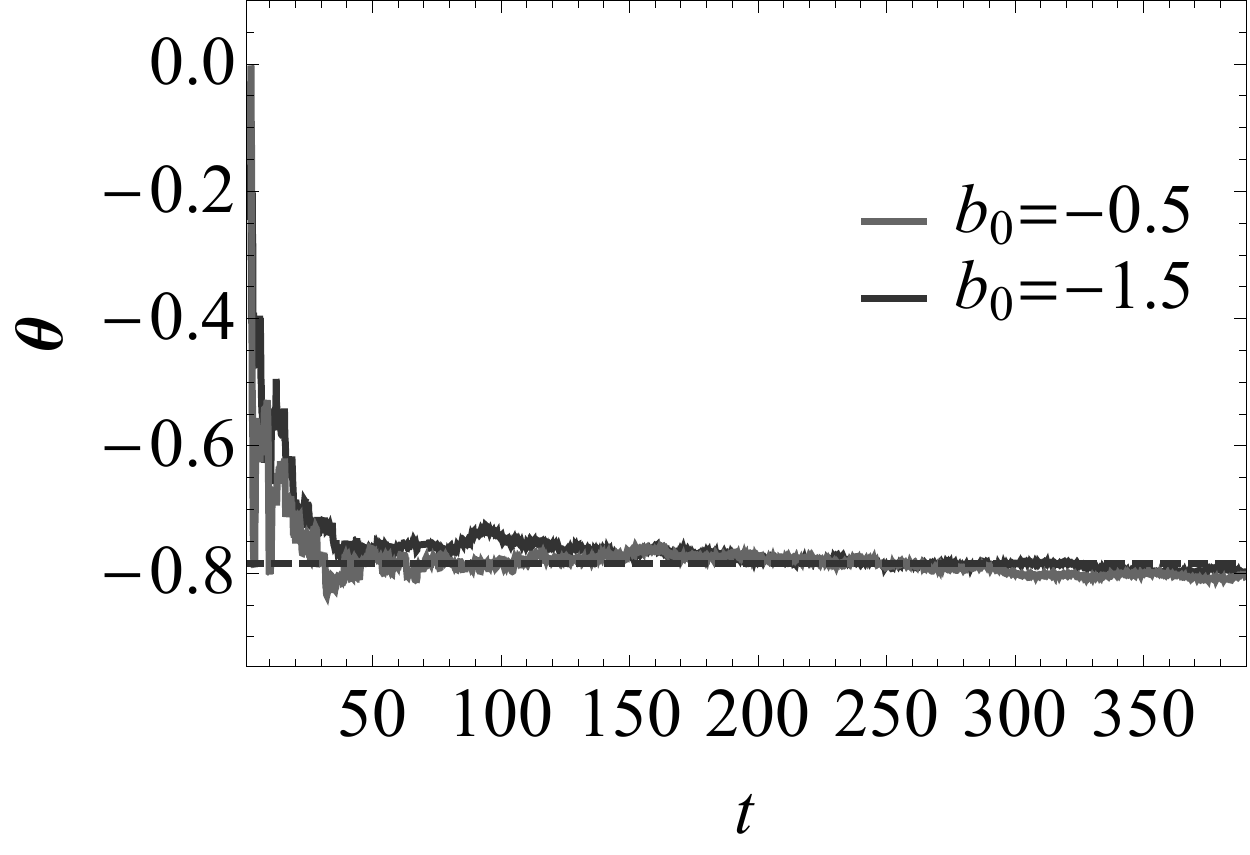}
\caption{\small The evolution of $|{\bf x}(t)|$ and $\theta(t)=\Phi(x_1(t),x_2(t))-S(t)/2$ for sample paths of solutions to system \eqref{Ex1} with $a_1=2$, $c_1=1$, $s_0=2$, $a_0=b_1=c_0=s_2=0$. The black dashed curve corresponds to $\theta(t)\equiv \phi_0$, where $\phi_0=-\pi/4$. In this case, $\lambda_1(\phi_0)<0$ if and only if $b_0<-1$.} \label{FigEx13}
\end{figure}

Note that in the case of $p=2$ the stochastic part of the decaying perturbation does not affect the stability of the equilibrium in system \eqref{Ex1}.

Thus, the stability of the equilibrium in \eqref{Ex1} depends both on the asymptotic regime in the system and on the perturbation frequency $s_0$. Moreover, the example from Section~\ref{PS} shows that stable phase locking of systems with an oscillatory stochastic perturbation is possible.

\subsection{Example 2}
Now, consider a nonlinear system
\begin{gather}\label{Ex2}
\begin{split}
&dx_1=x_2\,dt, \\
&dx_2=\left (-x_1+t^{-\frac{1}{2}}\frac{a(S(t))x_1^2x_2}{1+|{\bf x}|^2}+t^{-1}b(S(t))x_2\right)\,dt+t^{-\frac{1}{2}}c\big(S(t)\big)x_1\,dw_2(t), \quad t\geq 1,
\end{split}
\end{gather}
where 
\begin{gather*}
a(S)\equiv a_0+a_1\cos S,\quad 
b(S)\equiv b_0+b_1\cos S,\quad 
c(S)\equiv c_0+c_1\cos S,\quad 
S(t)\equiv  t+s_1t^{\frac{1}{2}}+s_2\log t
\end{gather*}
 with $p\in\mathbb Z_{+}$ and constant parameters $a_k$, $b_k$, $c_k$ and $s_k$. We see that system \eqref{Ex2} is of the form \eqref{FulSys} with $q=2$, 
\begin{gather*}
{\bf a}({\bf x},t)\equiv {\bf a}_0({\bf x})+t^{-\frac{1}{2}}{\bf a}_1({\bf x},S(t))+t^{-1}{\bf a}_2({\bf x},S(t)), \quad 
{\bf A}({\bf x},t)\equiv t^{-\frac{1}{2}}{\bf A}_1({\bf x},S(t)),\\
{\bf a}_0({\bf x})\equiv \begin{pmatrix}x_2\\-x_1\end{pmatrix}, \ \ 
{\bf a}_1({\bf x},S)\equiv \begin{pmatrix}0\\ \frac{a(S)x_1^2x_2}{1+|{\bf x}|^2}\end{pmatrix}, \ \ 
{\bf a}_2({\bf x},S)\equiv \begin{pmatrix}0\\ b(S)x_2\end{pmatrix}, \ \ 
{\bf A}_p({\bf x},S)\equiv \begin{pmatrix}0&0\\0& c(S)x_1\end{pmatrix}.
\end{gather*}
As in the previous example, the Hamiltonian of corresponding limiting system \eqref{LimSys} has the form $H(x_1,x_2)\equiv |{\bf x}|^2/2$. The change of variables described in Theorem~\ref{Th1} with $\ell=0$, $N=2$, $\xi_1(\phi,E)\equiv \sqrt{2E}\cos \phi$, $\xi_2(\phi,E)\equiv -\sqrt{2E}\sin \phi$,  $\tilde v_2(R,\theta,t)\equiv t^{-1/2}v_1(R,\theta,S(t))+t^{-1}v_2(R,\theta,S(t))$, $\tilde \psi_2(R,\theta,t)\equiv t^{-1/2}\psi_1(R,\theta,S(t))+ t^{-1}\psi_2(R,\theta,S(t))$
transforms system \eqref{Ex1} into the form \eqref{EQs} with 
\begin{align*}
\Lambda_1(v,\psi)=&\frac{a_0 v^3}{4 (1+2v^2)}, \\ 
\Omega_1(v,\psi)=&-\frac{s_1}{2}, \\
\Lambda_2(v,\psi)=&\frac{v}{96}\left(48 b_0+18c_0^2+9c_1^2+6c_1^2 \cos 2\psi -\frac{2a_1^2 v^4 \sin 2\psi}{(1+2v^2)^2}\right),\\
\Omega_2(v,\psi)=&-s_2-\frac{c_1^2\sin 2\psi}{16} \\
&-\frac{v^4 \left(165 a_0^2+104a_1^2-20 a_1^2 \cos2\psi+12 v^2(25 a_0^2 + 16 a_1^2 + 5 a_1^2 \cos 2 \psi )\right)}{960(1 + 2 v^2)^3},
 \end{align*}
It follows that conditions \eqref{LnOm}, \eqref{Lnas2} and \eqref{Omas} hold with $n=m=1$, $l=1$, $h=3$ and
\begin{align*}
&\lambda_{1,3}(\psi)\equiv \frac{a_0}{4}, \quad 
\lambda_{2}(\psi)\equiv \frac{48 b_0+18c_0^2+9c_1^2+6c_1^2 \cos 2\psi}{96}, \\ 
&\omega_{1,0}(\psi)\equiv -\frac{s_1}{2}, \quad 
\omega_{2,0}(\psi)\equiv -s_2-\frac{c_1^2\sin 2\psi}{16}.
\end{align*}
Moreover, assumption \eqref{classPert} holds with $\mu=|c_0|+|c_1|$ and $2p/q=1$.

Let $s_1\neq 0$. Then assumption \eqref{nzerom} holds and Theorem~\ref{Th5} is applicable with $m<q$. Hence, the phase drifting regime is realized, and the equilibrium $(0,0)$ is stable if $a_0<0$ and $b_0<-(6c_0^2+5c_1^2)/16$ (see Fig.~\ref{FigEx21}).

\begin{figure}
\centering
  \includegraphics[width=0.4\linewidth]{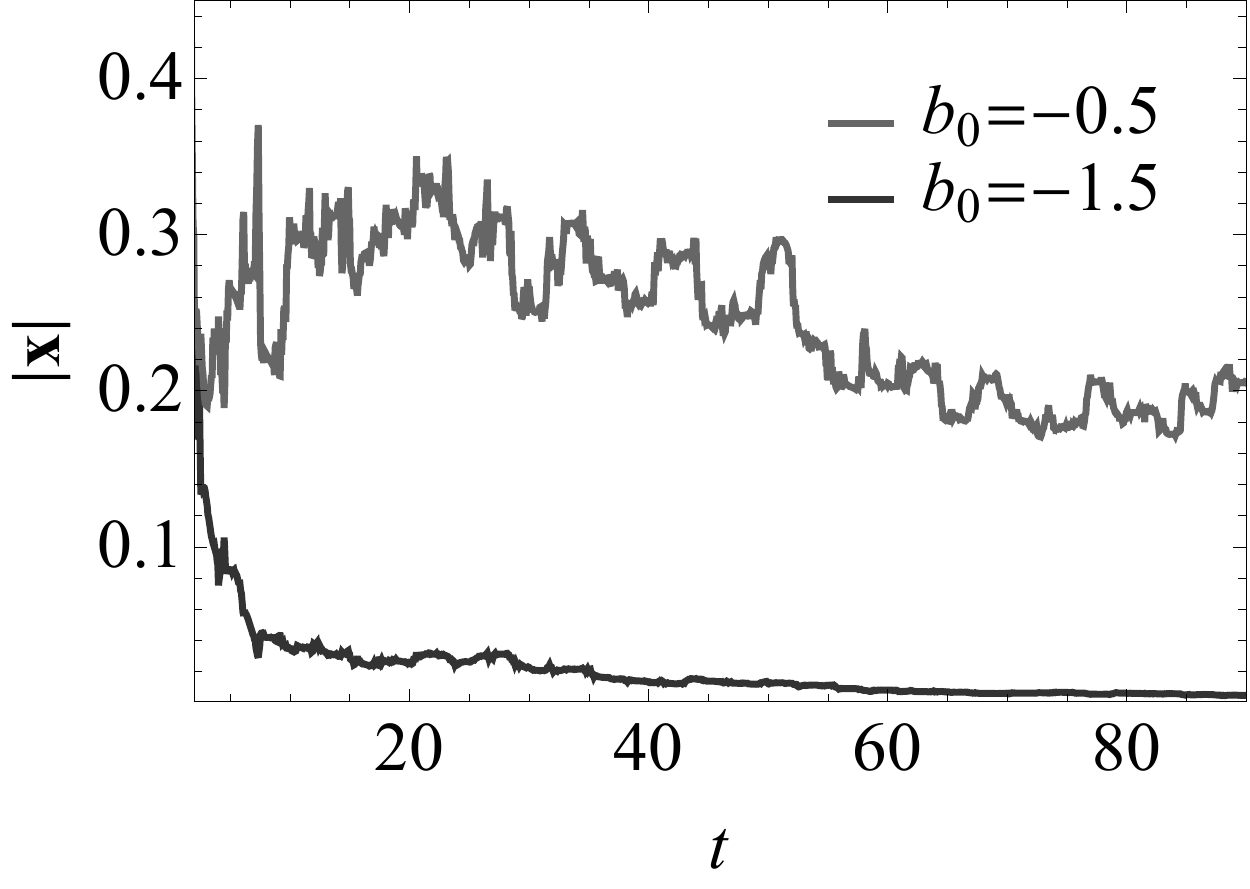}
\caption{\small The evolution of $|{\bf x}(t)|$  for sample paths of solutions to system \eqref{Ex2} with 
$a_0=-0.1$, 
$a_1= b_1=s_1=1$, 
$c_0=s_2=0$, 
$c_1=2$.} \label{FigEx21}
\end{figure}

Now let $s_1=0$. If $16|s_2|<c_1^2$, then the assumptions \eqref{LnOm} and \eqref{zerom} hold with $m=2$ and 
\begin{gather*} 
\phi_0= \frac{1}{2}\arcsin \left(-\frac{16 s_2}{c_1^2}\right)+\pi j, \quad 
j\in\mathbb Z, \quad 
\vartheta_2=-\frac{c_1^2}{4} \cos 2\phi_0<0,
\end{gather*}
and Theorem~\ref{Th3} is applicable with $n+l=m=q=2p$. Hence, the phase locking regime takes place (see~Fig.~\ref{FigEx22}), and the equilibrium $(0,0)$ is stable on exponentially long time interval as $\mu\to 0$ if 
\begin{gather*}
a_0<0, \quad 
b_0<b_\ast, \quad b_\ast:=-\frac{6c_0^2+3c_1^2+2\sqrt{c_1^4- (16 s_2)^2}}{16}.
\end{gather*}
\begin{figure}
\centering
 \includegraphics[width=0.4\linewidth]{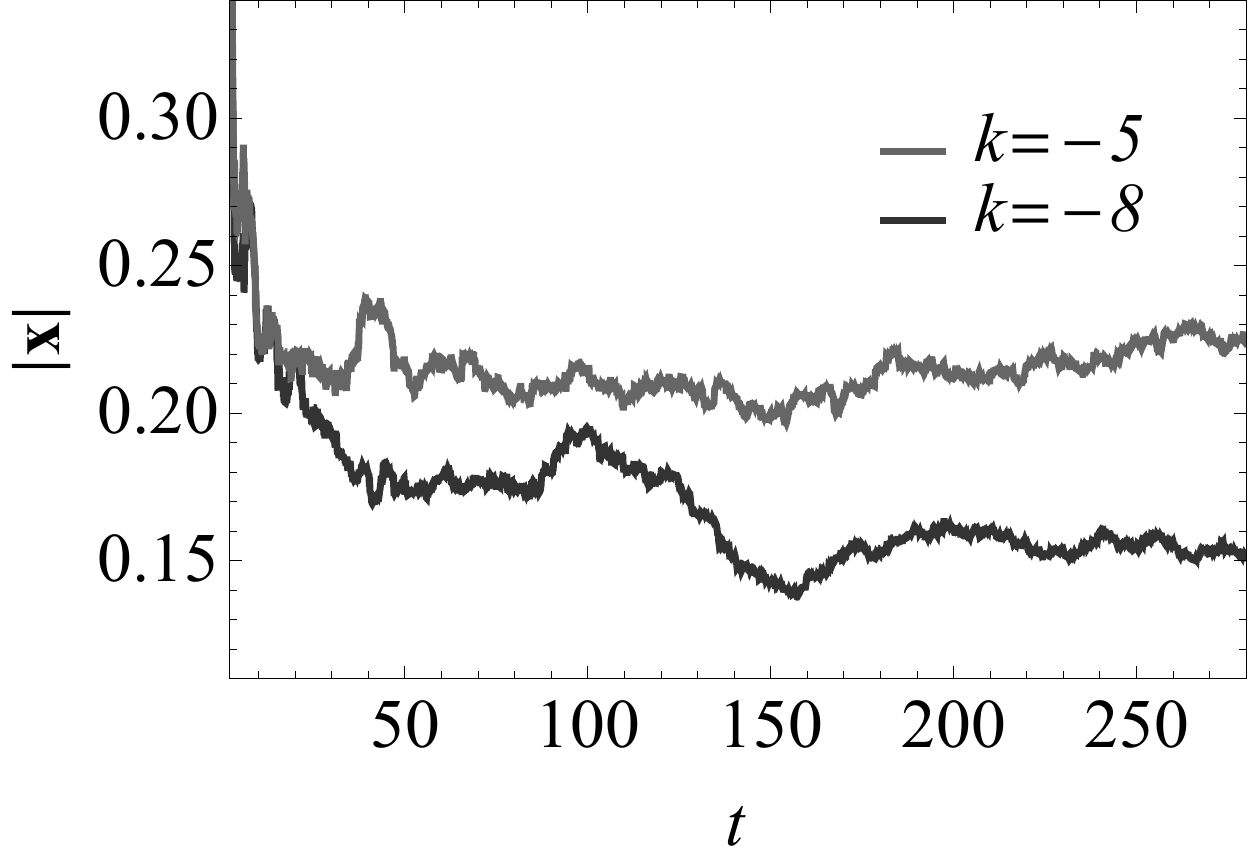}
\hspace{1ex}
  \includegraphics[width=0.4\linewidth]{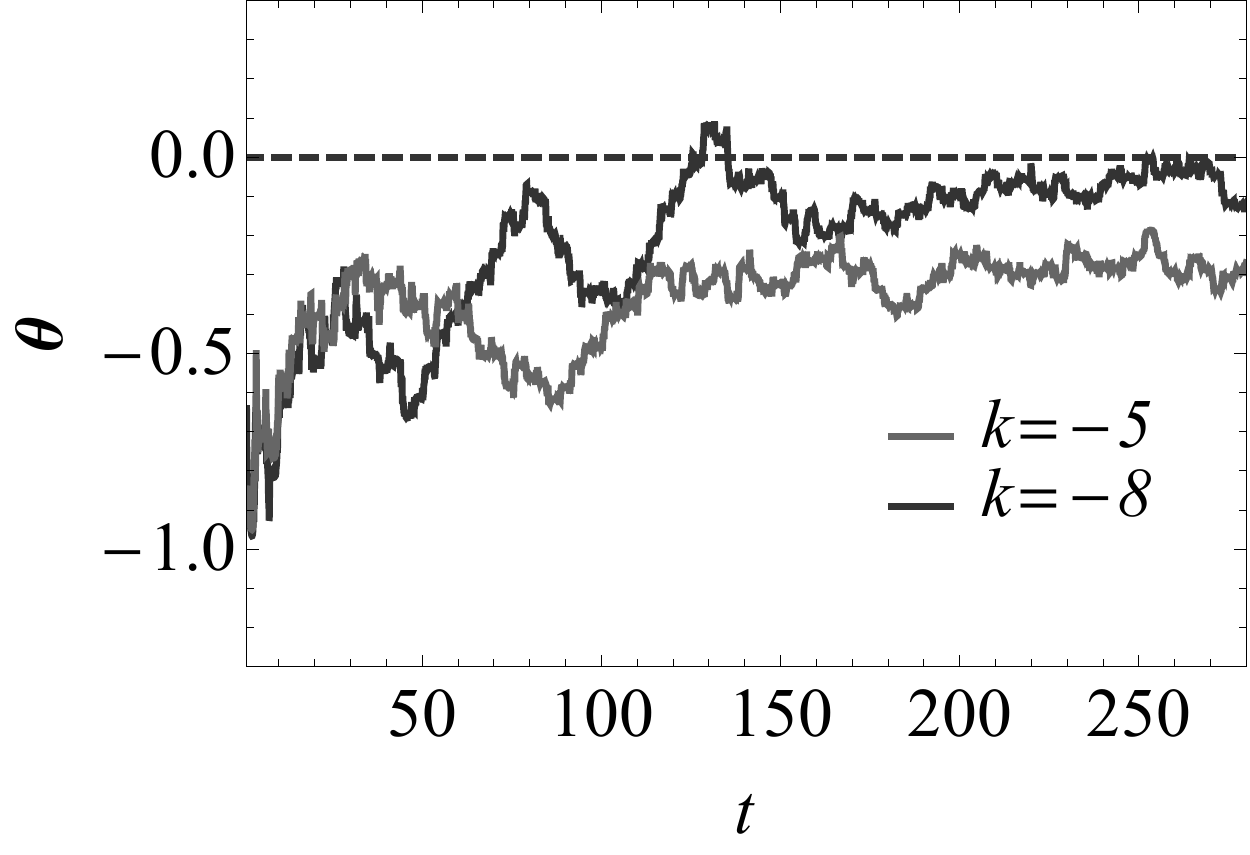}
\caption{\small The evolution of $|{\bf x}(t)|$ and $\theta(t)=\Phi(x_1(t),x_2(t))-S(t)$ for sample paths of solutions to system \eqref{Ex2} with 
$a_0=-1$, 
$a_1= b_1=1$, 
$c_0=s_1=s_2=0$, 
$b_0=k\mu^2/16$,
$c_1=\mu$, $\mu=0.5$.
The black dashed curve corresponds to $\theta(t)\equiv \phi_0$, where $\phi_0=0$. In this case, $b_\ast=-5\mu^2/16$.} \label{FigEx22}
\end{figure}
Moreover, form the second part of Theorem~\ref{Th3} it follows that if 
$a_0<0$ and $b_0>b_\ast$,
then the equilibrium $(0,0)$ is stable (see Fig.~\ref{FigEx23}) in the sense of \eqref{deftilde} with
\begin{gather*}
\tilde d({\bf x},t)\equiv \sqrt{t^{\frac{1}{2}}\left|\frac{|{\bf x}|}{\sqrt 2}-u_\varepsilon(t)\right|^2+|\Phi(x_1,x_2)-S(t)-\phi_\varepsilon(t)|^2},\\
u_\varepsilon(t)\approx t^{-\frac{1}{4}} u_0, \quad \phi_\varepsilon(t)\approx \phi_0, \quad t\to\infty, \quad u_0=\left[\frac{4\lambda_{2}(\phi_0)+1}{|a_0|}\right]^{\frac{1}{2}}.
\end{gather*}
\begin{figure}
\centering
 \includegraphics[width=0.4\linewidth]{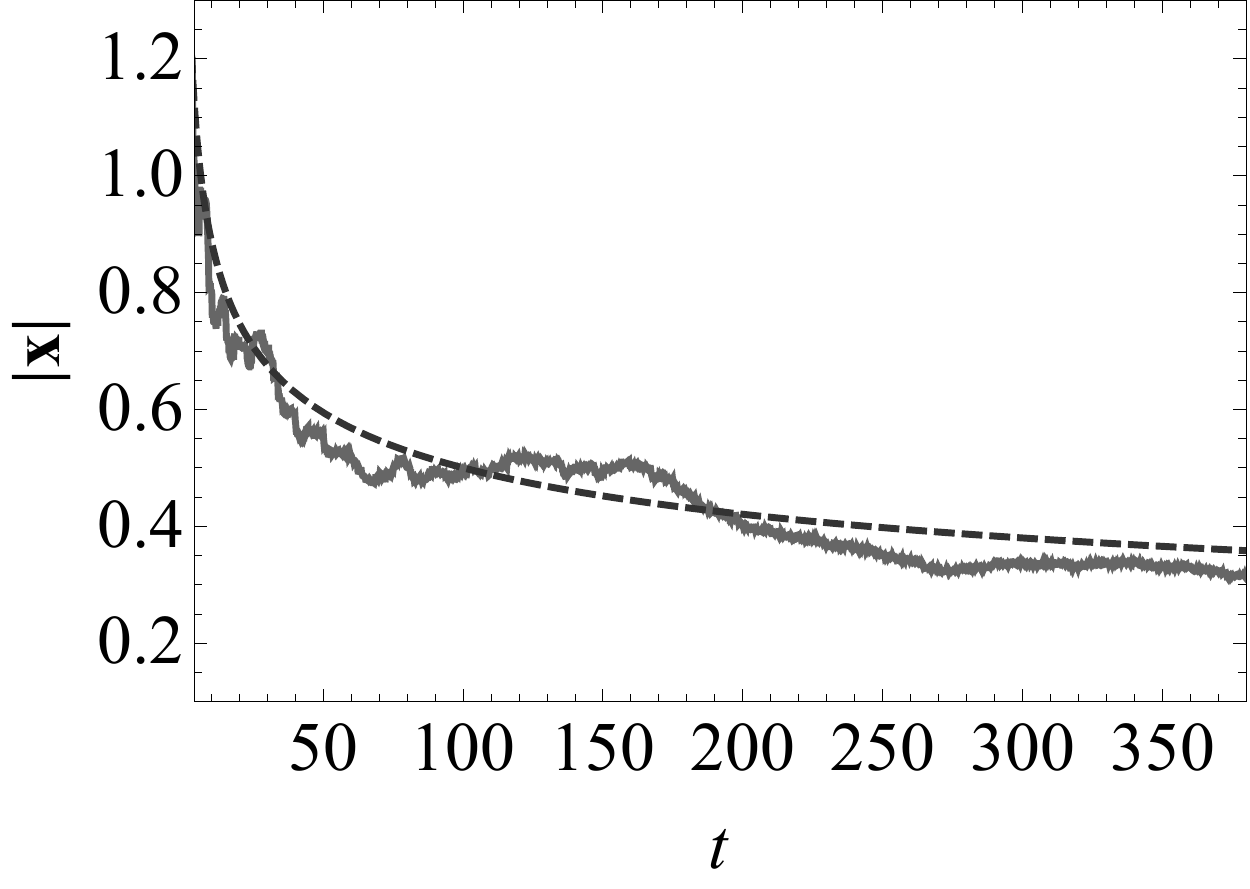}
\hspace{1ex}
 \includegraphics[width=0.4\linewidth]{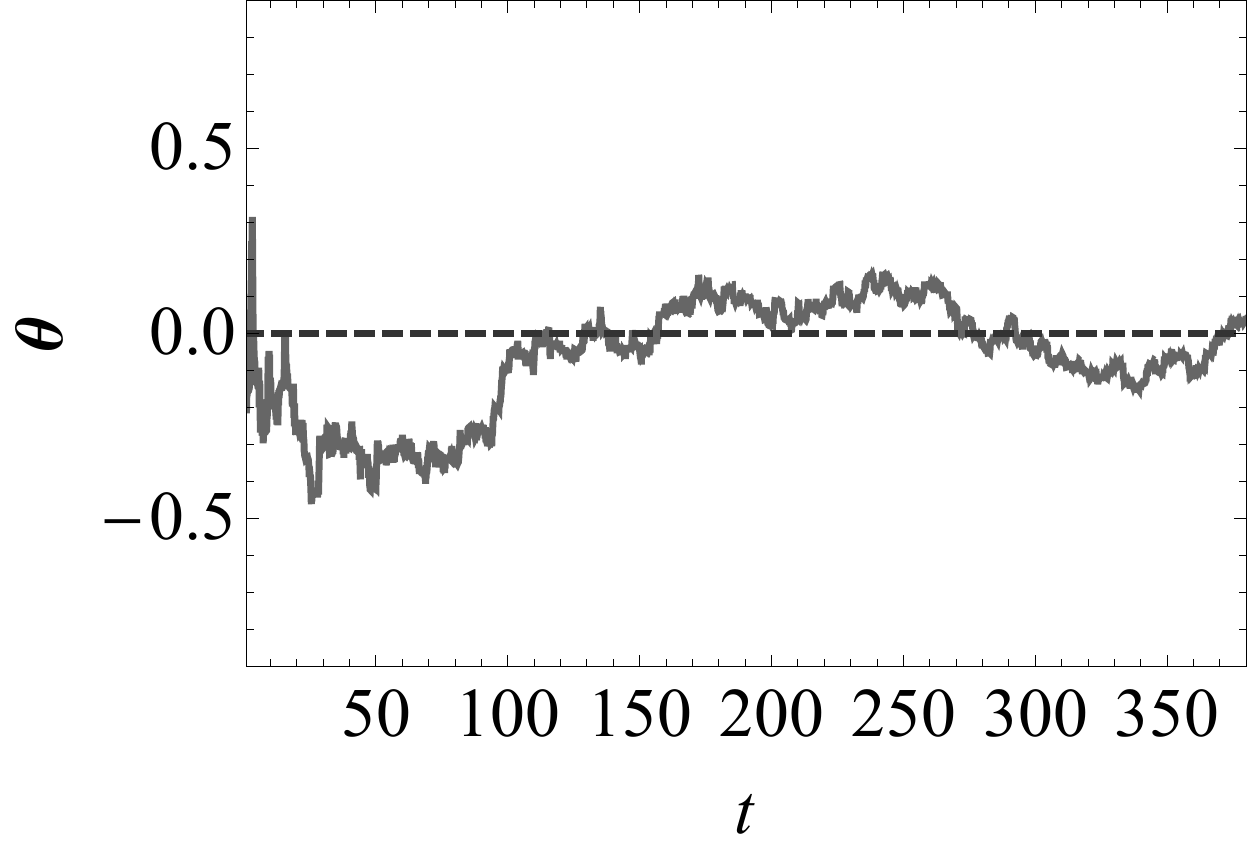}
\caption{\small The evolution of $|{\bf x}(t)|$ and $\theta(t)=\Phi(x_1(t),x_2(t))-S(t)$ for sample paths of solutions to system \eqref{Ex2} with 
$a_0=-1$, 
$a_1= b_1=1$, 
$c_0=s_1=s_2=0$, 
$b_0=3\mu^2/16$,
$c_1=\mu$, $\mu=0.5$.
The black dashed curves correspond to $|{\bf x}(t)|= t^{-1/4}u_0$ and $\theta(t)\equiv \phi_0$, where $u_0=\sqrt{5/2}$ and $\phi_0=0$.} \label{FigEx23}
\end{figure}

\section{Conclusion}
Thus, the influence of damped stochastic perturbations with oscillating coefficients on the stability of the equilibrium in asymptotically Hamiltonian systems in the plane has been investigated. We have shown that depending on the structure and the degree of perturbations damping, two asymptotic regimes for solutions near the equilibrium are possible: phase locking, when the phase of the perturbed system adjusts to the perturbation with a finite phase difference limit, and phase drifting, when the phase difference increases unlimitedly. The found stability conditions depend on the realized asymptotic regime. In particular, in the case of the phase locking regime, the stability of the equilibrium depends on the value of the phase difference limit. 

Note that the described stability conditions are only sufficient. Comparing these results with the conclusions for the corresponding truncated system \eqref{DetSys} indicates that the stochastic stability conditions are close to necessary. However, the problem of instability of the equilibrium for both asymptotic regimes has not been investigated in detail in this paper. This will be discussed elsewhere.  

\section*{Acknowledgements}

The research is supported by the Russian Science Foundation (Grant No. 23-11-00009).

\end{document}